\documentclass[a4paper,12pt]{amsart}

\usepackage[
  margin=30mm,
  marginparwidth=25mm,     
  marginparsep=2mm,       
  bottom=25mm,
  ]{geometry}
\usepackage{mathrsfs}
\usepackage[latin1]{inputenc}
\usepackage[english]{babel}
\usepackage{amsmath,amssymb,amsthm}
\usepackage{latexsym}
\usepackage{delarray}
\usepackage{bbm}
\usepackage{hyperref}
\usepackage[pdftex,usenames,dvipsnames]{color}
\usepackage{bbding,trfsigns}
\usepackage{datetime}
\usepackage{mathrsfs}
\usepackage{enumitem}
\usepackage{tikz}

\newtheorem{Th}{Theorem}[section]
\newtheorem{Prop}[Th]{Proposition}
\newtheorem{Lem}[Th]{Lemma}
\newtheorem{Cor}[Th]{Corollary}
\newtheorem{Rem}[Th]{Remark}
\newtheorem{defn}[Th]{Definition}

\newcommand{\N}{\mathbb{N}}
\newcommand{\R}{\mathbb{R}}

\newcommand{\Rn}{\mathbb{R}^{n}}
\newcommand{\Z}{\mathbb{Z}}

\renewcommand{\d}{\partial}
\renewcommand{\SS}{{\mathscr{S}}}
\def\bra#1{{\langle{#1}\rangle}}
\def\la{\langle}
\def\ra{\rangle}

\newcommand{\eps}{\varepsilon}
\newcommand{\dd}{\, \mathrm{d}}
\newcommand{\ddd}{\,\text{\rm{\mbox{\dj}}}}

\newcommand{\at}{\mathfrak{a}}

\newcommand{\brkt}[1]{\Big({#1}\Big)}
\newcommand{\set}[1]{\Big\{{#1}\Big\}}
\newcommand{\norm}[1]{\Big\Vert#1\Big\Vert}
\newcommand{\jap}[1]{\left\langle {#1}\right \rangle }
\newcommand{\abs}[1]{\left |{#1}\right |}

\newcommand{\m}[1]{\begin{equation*}
\begin{split}
#1
\end{split}
\end{equation*}}
\newcommand{\nm}[2]{\begin{equation}\label{#1}
\begin{split}
#2
\end{split}
\end{equation}}

\DeclareMathOperator{\supp}{supp}

\allowdisplaybreaks 

\title[On the $L^p$-boundedness of general FIOs]
{Regularity of Fourier integral operators with amplitudes in general H\"ormander classes}

\author[A. J. Castro]{Alejandro J. Castro}
\author[A. Israelsson]{Anders Israelsson}
\author[W. Staubach]{Wolfgang Staubach}

\address{\newline
       Alejandro J. Castro \newline
       Department of  Mathematics, Nazarbayev University, \newline
		010000 Nur-Sultan, Kazakhstan}
\email{alejandro.castilla@nu.edu.kz}

\address{\newline
       Anders Israelsson, Wolfgang Staubach \newline
       Department of  Mathematics, Uppsala University, \newline
       S-751 06 Uppsala, Sweden}
       \email{Anders.Israelsson@math.uu.se, wulf@math.uu.se}

 \thanks{
 The first author is supported by the Nazarbayev University Faculty Development Competitive Research Grants Program, grant number 110119FD4544.
 During the initial stage of preparation of the paper, the third author was partially supported by a grant from the Crafoord foundation.}



\begin{document}


\maketitle

\begin{abstract}
We prove the global $L^p$-boundedness of Fourier integral operators that model the parametrices for hyperbolic  partial differential equations, with amplitudes in classical H\"ormander classes $S^{m}_{\rho, \delta}(\R^n)$ for parameters $0<\rho\leq 1$, $0\leq \delta<1$. We also consider the regularity of operators with amplitudes in the exotic class $S^{m}_{0, \delta}(\R^n)$, $0\leq \delta< 1$ and the forbidden class $S^{m}_{\rho, 1}(\R^n)$, $0\leq\rho\leq 1.$   Furthermore we show that despite the failure of the $L^2$-boundedness of operators with amplitudes in the forbidden class $S^{0}_{1, 1}(\R^n)$, the operators in question are bounded on Sobolev spaces $H^s(\R^n)$ with $s>0.$ This result extends those of Y. Meyer and E. M. Stein to the setting of Fourier integral operators.
\end{abstract}

\hspace*{0.75cm}\\
\tableofcontents


\section{Introduction}

In this paper we investigate the local and global  regularity of Fourier integral operators (FIOs) of the form
\begin{equation*}
	T_a^\varphi f(x) = \frac{1}{(2\pi)^n} \int_{\R^n} e^{i\varphi(x,\xi)}\,a(x,\xi)\,\widehat f (\xi) \dd\xi,
	\end{equation*}
with amplitudes in H\"ormander classes $S^m_{\rho, \delta}(\R^n)$ consisting of functions in $\mathcal{C}^{\infty}(\R^n \times \R^n)$ satisfying \begin{equation*}
	\left|\partial_{\xi}^{\alpha}\partial_{x}^{\beta} a(x,\xi)\right |
	\leq C_{\alpha , \beta} (1+|\xi|)^{m-\rho|\alpha|+\delta|\beta|},
\end{equation*}	
for $0\leq\rho\leq 1$, $0\leq \delta\leq 1$. More specifically we  consider the boundedness in $L^p(\R^n)$ and $H^s(\R^n)$ (Sobolev spaces) for FIOs that model the parametrices of variable coefficient wave equations where the $\mathrm{rank}\,(\partial^2_{\xi \xi}\varphi(x,\xi))=n-1$. The corresponding investigation for the other extreme case, i.e. $\mathrm{rank}\,(\partial^2_{\xi \xi}\varphi(x,\xi))=0$ which is the pseudodifferential operator-case, was carried out by J. Alvarez, and J. Hounie in \cite{AH}. In that paper the authors consider the $L^p$-boundedness of pseudodifferential operators with symbols in $S^m_{\rho, \delta}(\Rn)$ where $0<\rho\leq 1$, $0\leq \delta < 1$. In this paper we also consider the case of operators with exotic amplitudes i.e. those with amplitudes in $S^m_{0,\delta}(\R^n)$, $0\leq \delta<1$ (which was missing in \cite{AH}) and thereby complete the picture regarding the $L^p$-boundedness of pseudodifferential operators. The case of $\rho=\delta=0$ was treated (using methods that are different than ours) by R. Coifman and Y. Meyer in {\cite{CM}}, see also \cite{Stein}.\\

Prior to this investigation, the only source for results regarding $L^p$-regularity of FIOs in $S^m_{\rho, \delta}$-classes were those by A. Seeger. C. Sogge and E. M. Stein \cite{SSS}, where the authors established the local $L^p$-boundedness for $\rho\in [1/2,1]$ and $\delta=1-\rho.$ \\

Regarding global $L^p$-boundedness, the results of M. Ruzhansky and M. Sugimoto \cite{RuSu} are global extensions of those of Seeger-Sogge-Stein, however they are confined to the amplitudes with $\rho=1$ and $\delta=0$.\\

If one goes outside the aforementioned H\"ormander classes of operators, then global boundedness results have been proven in various settings for example in the papers by S. Coriasco and M. Ruzhansky {\cite{CR1}}, E. Cordero, F. Nicola and L. Rodino \cite{CNR1,CNR2}. Recently, A. Hassell, P. Portal and J. Rozendaal \cite{HPR} obtained results regarding global boundedness of Fourier integral operators, that go beyond those in \cite{RuSu}. More precisely in \cite{HPR} the authors also establish the regularity of FIOs with amplitudes that decay faster than those in $S^m_{\rho, 1-\rho}(\R^n)$ (with $\rho\in [1/2,1]$), when differentiated in the radial direction in the frequency variables. In \cite{DS} D. Dos Santos Ferreira and W. Staubach considered amplitudes in very rough classes (that also contain all the H\"ormander classes $S^m_{\rho, \delta}(\R^n)$), and proved global $L^p$-boundedness of corresponding FIOs. However, due to the roughness of the amplitudes, the order $m$ (which depends on $\rho$ and $\delta$) is not as good as the expected one for smooth amplitudes, and further work is needed to achieve the right order of decay required for the $L^p$-boundedness of, for example, FIOs that yield parametrices for variable coefficient wave equations.\\

As one of the justifications of this investigation, we would like to mention that the work of R. Melrose and M. Taylor \cite{MT}, and also the study of FIOs on certain nilpotent Lie groups (other than the Heisenberg group) motivates the consideration of FIOs with amplitudes in H\"ormander classes $S^m_{1/3, 2/3}(\R^n)$, for which, so far, no $L^p$-boundedness results have been available. Regarding $L^2$-boundedness of operators with general H\"ormander-class amplitudes, in \cite{DS} Dos Santos Ferreira and Staubach showed that $T^\varphi_a$ is globally $L^2$-bounded, provided that $\rho, \delta\in [0,1]$, $\delta\neq 1$ and $m=\min(0,n(\rho-\delta)/2)$, or $\rho\in [0,1],$ $\delta=1$ and $m<n(\rho-1)/2$. This result is  sharp. In this paper we also discuss the global $L^p$-boundedness of FIOs with forbidden amplitudes $S^m_{\rho,1}(\R^n)$, $0\leq \rho\leq 1$. For instance, it was shown in \cite[Theorem 2.17]{DS} that FIOs with strongly non-degenerate phase functions and amplitudes in $S^m_{1,1}(\R^n)$ are $L^p$-bounded if and only if $m<-(n-1)|1/p-1/2|,$ a result which parallels the well-known facts about pseudodifferential operators with forbidden symbols. However, the endpoint case of $S^0_{1,1}(\Rn)$, which is not covered by the results above, is of particular interest. Indeed as Y. Meyer {\cite{Meyer}} and E. Stein (unpublished) have shown, despite the lack of, say $L^2$-boundedness ({\cite[Prop. 2, p. 272]{Stein}}), the pseudodifferential operators with symbols in $S^0_{1,1}(\Rn)$ map Sobolev spaces $H^s(\R^n)$ with $s>0$ continuously to themselves. This remarkable fact has had a large impact on the applications of J. M. Bony's paradifferential calculus {\cite{Bony}} to a systematic study of various nonlinear partial differential equations, see \cite{Taylor} for a comprehensive presentation. In this paper we will also prove that FIOs with amplitudes in $S^0_{1,1}(\R^n)$ yield bounded operators on Sobolev spaces with positive exponents.\\

The paper is organised as follows; in Section \ref{prelim} we recall some definitions, facts and results from microlocal and harmonic analysis that will be used throughout the paper. In Section \ref{RS globalisation} we reduce the FIOs to a form that is amenable for Ruzhansky-Sugimoto's globalisation technique, which will in turn be adapted to general classes of H\"ormander-class amplitudes. In Section \ref{Composition FIOs} we first prove a general composition formula for the left-action of a Fourier multiplier on an FIO with amplitude in general H\"ormander classes. Our result extends the known results to the global setting and all values of $\rho, \delta$ (although the case of $\delta=1$ has to be excluded). Thereafter, in Section \ref{SSS decomposition}, we extend the method of Seeger-Sogge-Stein to the case of FIOs with general classical H\"ormander-class amplitudes, and decompose the Fourier integral operators into certain pieces for which we establish the basic kernel estimates. In Section \ref{Lp section} we prove our main $L^p$-boundedness theorems for FIOs with amplitudes in the  exotic, classical and forbidden H\"ormander classes. Finally in Section \ref{Sobolev bounedness} we prove the $H^s$-boundedness for FIOs with amplitudes in the forbidden class $S^0_{1,1}(\R^n)$ for $s>0$ and thereby extend the result of Meyer and Stein to the FIO-setting. Indeed we produce a result in a more general class of amplitudes $C^{r}_{*}S^0_{1,1}(\Rn)$ with contains the class $S^0_{1,1}(\R^n)$.

\section{Preliminaries}\label{prelim}

As is common practice, we will denote positive constants in the inequalities by $C$, which can be determined by known parameters in a given situation but whose
value is not crucial to the problem at hand. Such parameters in this paper would be, for example, $m$, $p$,  $s$, $n$,  and the constants connected to the seminorms of various amplitudes or phase functions. The value of $C$ may differ
from line to line, but in each instance could be estimated if necessary. We also write $a\lesssim b$ as shorthand for $a\leq Cb$ and moreover will use the notation $a\approx b$ if $a\lesssim b$ and $b\lesssim a$.\\
\begin{defn}\label{def:LP}
Let $\psi_0 \in \mathcal C_c^\infty(\R^n)$ be equal to $1$ on $B(0,1)$ and have its support in $B(0,2)$. Then let
$$\psi_j(\xi) := \psi_0 \left (2^{-j}\xi \right )-\psi_0 \left (2^{-(j-1)}\xi \right ),$$
where $j\geq 1$ is an integer and $\psi(\xi) := \psi_1(\xi)$. Then $\psi_j(\xi) = \psi\left (2^{-(j-1)}\xi \right )$ and one has the following Littlewood-Paley partition of unity
\m{
\sum_{j=0}^\infty \psi_j(\xi) = 1, \quad \text{\emph{for all }}\xi\in\R^n .
}

\noindent It is sometimes also useful to define a sequence of smooth and compactly supported functions $\Psi_j$ with $\Psi_j=1$ on the support of $\psi_j$ and $\Psi_j=0$ outside a slightly larger compact set. One could for instance set
\m{
\Psi_j := \psi_{j+1}+\psi_j+\psi_{j-1},
}
with $\psi_{-1}:=\psi_{0}$.
\end{defn} \hspace*{1cm}\\
In what follows we define the Littlewood-Paley operators by
\[
    \psi_j(D)\, f(x)= \int_{\R^n}   \psi_j(\xi)\,\widehat{f}(\xi)\,e^{ix\cdot\xi}\, \ddd\xi,
\]
where $\ddd \xi$ denotes the normalised Lebesgue measure ${\dd \xi}/{(2\pi)^n}$ and\[
	\widehat{f}(\xi)=\int_{\R^n} e^{-i x\cdot\xi}\,f(x) \dd x,
\]
is the Fourier transform of $f$.
\noindent Using the Littlewood-Paley decomposition of Definition \ref{def:LP}, we define the \emph{Sobolev space} $H^s(\Rn)$ in a somewhat unusual way. One can however show that this is equivalent to the standard definition of $H^s(\Rn)$.
\begin{defn}\label{def:Sobolev}
	Let $s \in {\mathbb R}$. The Sobolev space is defined by
	\[
	{H}^s(\R^n)
	:=
	\Big\{
	f \in {\mathscr{S}'}(\R^n) \,:\,
	\|f\|_{{H}^s(\R^n)}
	:=
	\Big(
	\sum_{j=0}^\infty
	4^{js}\|\psi_j(D)f\|^{2}_{L^2(\R^n)}
	\Big)^{1/2}<\infty
	\Big\},
	\]
where $\mathscr{S}'(\R^n)$ denotes the space of tempered distributions.	
\end{defn}

\begin{Rem}\label{rem:Sobolev}
Different choices of the basis $\{\psi_j\}_{j=0}^\infty$ give equivalent norms of $H^s(\R^n)$ in \emph{Definition \ref{def:Sobolev},} see e.g. \cite{Trie83}. We will use either $\{\psi_j\}_{j=0}^\infty$ or  $\{\Psi_j\}_{j=0}^\infty$ to define the norm of $H^s(\R^n)$.
\end{Rem}
\begin{Rem}\label{rem:sobolevnormer}
By Fubini's theorem, one can change the order of the norms in \emph{Definition \ref{def:Sobolev}}, i.e.
\m{\|f\|_{{H}^s(\R^n)}
\approx
\Big\|\Big\{\sum_{j=0}^{\infty} 4^{j s}\left|\psi_j(D) f\right|^{2}\Big\}^{1 / 2}\Big\|_{L^{2}(\R^n)}
\approx	
\Big(
	\sum_{j=0}^\infty
	4^{j s}\|\psi_j(D) f\|^{2}_{L^2(\R^n)}
	\Big)^{1/2} .
	}
\end{Rem}
	
Also, using fairly standard Littlewood-Paley theory one can show the following well-known result:
\begin{Lem}\label{lem:generalisedTL}
Let $\{f_j\}_{j=0}^\infty \subset \SS'(\Rn)$  be such that
$$\supp \widehat f_j \subseteq \{ \xi\in\Rn: \abs\xi \lesssim 2^{j}\}, \quad j\geq0.$$
Then, for $s > 0$, one has
\m{
\norm{\sum_{j=0}^\infty f_j}_{H^s(\Rn)}
\lesssim
\norm{ \set{\sum_{j=0}^\infty 4^{js}\abs{f_j}^2}^{1/2}}_{L^2(\Rn)},
}
\end{Lem}
For a proof, see e.g. \cite{Taylor}.\\

In proving the $L^p$-boundedness of FIOs ($1<p<\infty$), the standard procedure is to first show the boundedness of the operator (and its adjoint) from the Hardy space $\mathscr{H}^1(\R^n)$ to $L^1(\R^n)$ and thereafter interpolate the results with the $L^2$-boundedness. In proving the Hardy space boundedness, the main tool is to use the so-called Hardy space atoms.
\begin{defn}\label{def:Hpatom}
Let $p\in (0,1].$ A function $\at$ is called an $\mathscr{H}^p$-atom if for some $x_0\in \R^n$ and $r>0$ the following three conditions are satisfied:
\begin{enumerate}
\item[$i)$] $\supp \at\subset B(x_{0}, r)$,
\item[$ii)$] $ |\at(x)|\leq|B(x_{0}, r)|^{-1/p},$
\item[$iii)$] \smallskip $ \int_{\R^n} x^{\alpha}\, \at(x)\dd x=0$ for all $|\alpha|\leq N$ for some $N\geq  n(1/p-1)$.
\end{enumerate}
Then a distribution $f\in \mathscr{H}^p (\R^n)$, has an atomic decomposition
$$f=\sum_{j=0}^\infty\lambda_{j}\at_{j},$$
where the $\lambda_{j}$ are constants such that
$$ \sum_{j=0}^\infty|\lambda_{j}|^p
\approx
\|f\|^p_{\mathscr{H}^p(\mathbb{R}^{n})}$$
and the $\at_{j}$ are $\mathscr{H}^p$-atoms.
\end{defn}
\begin{Rem}
Different choices of $N$ in $iii)$ above give equivalent definitions of the $\mathscr{H}^p$-norm.
\end{Rem}

Next we define the building blocks of the FIOs and the pseudodifferential operators. These are the amplitudes (symbols in the pseudodifferential setting) and the phase functions. The class of amplitudes considered in this paper were first introduced by L. H\"ormander in \cite{Hor1}.
\begin{defn}\label{symbol class Sm}
Let $m\in \R$ and $\rho, \delta \in [0,1]$. An \textit{amplitude} \emph{(}symbol\emph{)} $a(x,\xi)$ in the class $S^m_{\rho,\delta}(\R^n)$ is a function $a\in \mathcal{C}^\infty (\R ^n\times \R ^n)$ that verifies the estimate
\m{
\left |\partial_\xi^\alpha \partial_x^\beta a(x,\xi) \right |\lesssim \langle\xi\rangle ^{m-\rho|\alpha|+\delta\abs\beta},
}
for all multi-indices $\alpha$ and $\beta$ and $(x,\xi)\in \R ^n\times \R ^n$, where
$\langle\xi\rangle:= (1+|\xi|^2)^{1/2}.$
We shall henceforth refer to $m$ as the order of the amplitude. Following the folklore in harmonic and microlocal analysis, we shall refer to the class $S_{0,\delta}^m(\Rn)$ as the exotic class and to $S_{\rho,1}^m(\Rn)$ as the forbidden class of amplitudes.
\end{defn}
{Towards the end of this paper, in connection with the amplitudes with low spatial regularity and also the forbidden amplitudes, we will use the \emph{Zygmund class} $C_{*}^{r} (\R^n)$ whose definition we now recall.
\begin{defn}\label{def:Zygmund}
	Let $r \in \R$. The Zygmund class is defined by
	\[
	C_{*}^{r} (\R^n)
	:=
	\Big\{
	f \in {\mathscr{S}'}(\R^n) \,:\,
	\|f\|_{C_{*}^{r} (\R^n)}
	:=
	\sup_{j \geq 0}2^{jr}\|\psi_j(D)f\|_{L^\infty(\R^n)}
	<\infty\Big\}.
	\]
\end{defn}
If $C^r(\R^n)$, $r\in \R_+$, denotes the H\"older space, and $\mathcal{C}^r(\R^n)$ denotes the space of continuous functions with continuous derivatives of orders up to and including $r$, then one also has that
\begin{equation}\label{Zygmundegenskap}
   C_{*}^{r}(\R^n) =C^r(\R^n)  \quad \text{for}\,\,\, r\in \R_+ \setminus \Z_+ \quad \text{and}\quad  \mathcal{C}^r(\R^n) \subset C_{*}^{r}(\R^n) \quad \text{for}\, r\in \Z_+.
\end{equation}
}
\noindent In connection to the definition of the Zygmund class, there is another class of amplitudes which have low regularity in the $x$-variable, which
were considered by G. Bourdaud in {\cite{Bourd}}.

\begin{defn}\label{symbolclass ZygSm}
\noindent Let $m\in \R$, $0\leq \delta\leq 1$ and $r>0$. An \textit{amplitude} \emph{(}symbol\emph{)} $a(x,\xi)$ is in the class $C_{*}^{r} S_{1, \delta}^{m}(\R^n)$  if it is  $\mathcal{C}^\infty ( \R ^n)$ in the $\xi$ variable and verifies the estimates
\m{
\| \partial_\xi^\alpha a(\cdot,\xi) \|_{L^\infty(\R^n)}\lesssim \langle\xi\rangle ^{m-|\alpha|},
}
and
\m{
\| \partial_\xi^\alpha a(\cdot,\xi)\|_{C_{*}^{r}(\R^n)}\lesssim \langle\xi\rangle ^{m-|\alpha|+\delta r},
}
for all multi-indices $\alpha$ and $\xi\in \R ^n$. Here $C_{*}^{r}(\R^n)$ is the Zygmund class of \emph{Definition \ref{def:Zygmund}}.

\end{defn}
It is important to note that  $S^m_{1,1} (\R^n)\subset C_{*}^{r} S_{1,1}^{m}(\R^n),$ for all $r>0$, which follows from \eqref{Zygmundegenskap}.\\

\noindent Given the symbol classes defined above, one associates to the symbol its \textit{Kohn-Nirenberg quantisation }as follows:
\begin{defn}
Let $a$ be a symbol. Define a pseudodifferential operator \emph{(}$\Psi\mathrm{DO}$ for short\emph{)} as the operator
\begin{equation*}
a(x,D)f(x) := \int_{\R^n}e^{ix\cdot\xi}\,a(x,\xi)\,\widehat{f}(\xi) \ddd \xi,
\end{equation*}
a priori defined on the Schwartz class $\mathscr{S}(\R^n).$
\end{defn}\hspace*{1cm}\\
\noindent In order the define the Fourier integral operators that are studied in this paper, following \cite{DS}, we also define the classes of phase functions.
\begin{defn}\label{def:phi2}
A \textit{phase function} $\varphi(x,\xi)$ in the class $\Phi^k$ is a function \linebreak$\varphi(x,\xi)\in \mathcal{C}^{\infty}(\R^n \times\R^n \setminus\{0\})$, positively homogeneous of degree one in the frequency variable $\xi$ satisfying the following estimate

\begin{equation}\label{C_alpha}
	\sup_{(x,\,\xi) \in \R^n \times\R^n \setminus\{0\}}  |\xi| ^{-1+\vert \alpha\vert}\left | \partial_{\xi}^{\alpha}\partial_{x}^{\beta}\varphi(x,\xi)\right |
	\leq C_{\alpha , \beta},
	\end{equation}
	for any pair of multi-indices $\alpha$ and $\beta$, satisfying $|\alpha|+|\beta|\geq k.$
    In this paper we will mainly use  phases in class $\Phi^2$ and $\Phi^1$.
\end{defn}\hspace*{1cm}

We will also need to consider phase functions that satisfy  certain {\em non-degeneracy conditions}. These conditions have to be adapted to the case of local and global boundedness in an appropriate way. Following \cite{SSS}, in connection to the investigation of the local results, that is, under the assumption that the $x$-support of the amplitude $a(x,\xi)$ lies within a fixed compact set $\mathcal{K}$, the non-degeneracy condition is formulated as follows:
\begin{defn}\label{nondegeneracy}
Let $\mathcal{K}$ be a fixed compact subset of $\R^n$. One says that the phase function $\varphi(x,\xi) $  satisfies the non-degeneracy condition if
\begin{equation*}
	\det \brkt{\partial^{2}_{x_{j}\xi_{k}}\varphi(x,\xi)}\neq 0,\qquad \mbox{for all $(x,\xi)\in \mathcal{K}\times \R^n\setminus\{0\}$}.
\end{equation*}
\end{defn}\hspace*{1cm}\\
\noindent Following the approaches in e.g.  \cite{DS, RLS, RuSu}, for the global $L^p$-boundedness  results that were established in those papers, we also define the following somewhat stronger notion of non-degeneracy:
\begin{defn}
One says that the phase function $\varphi(x,\xi)$ satisfies the strong non-degeneracy condition \emph{(}or $\varphi$ is $\mathrm{SND}$ for short\emph{)} if
\begin{equation}\label{eq:SND}
	\Big |\det \brkt{\partial^{2}_{x_{j}\xi_{k}}\varphi(x,\xi)} \Big |
	\geq \delta,\qquad \mbox{for  some $\delta>0$ and all $(x,\xi)\in \mathbb{R}^{n} \times \R^n\setminus\{0\}$}.
\end{equation}
\end{defn}
\hspace*{1cm}\\
 Having the definitions of the amplitudes and the phase functions at hand, one has
\begin{defn}\label{def:FIO}
	A Fourier integral operator  \emph($\mathrm{FIO}$ for short\emph) $T_a^\varphi$ with amplitude $a$ and phase function $\varphi$, is an operator defined \emph{(}once again a-priori on $\mathscr{S}(\R^n)$\emph{)} by
	\begin{equation}\label{eq:FIO}
	T_a^\varphi f(x) := \int_{\R^n} e^{i\varphi(x,\xi)}\,a(x,\xi)\,\widehat f (\xi) \ddd\xi,
	\end{equation}
	where $\varphi(x,\xi)\in\mathcal{C}^{\infty}(\R^n \times \R^n\setminus\{0\})$ and is positively homogeneous of degree one in $\xi$.
\end{defn}

In this paper, the basic $L^2$-boundedness result which we shall utilise for the FIOs, is the following proposition which could be found in \cite{DS} as Theorems 2.2 and 2.7.

\begin{Prop}\label{basicL2}
Let  $\rho, \delta\in [0,1]$, $\delta\neq 1$. Assume that $a(x,\xi)\in S^{m}_{\rho,\delta}(\R^n)$ and $\varphi(x,\xi)$ is in the class $\Phi^2$ and is \emph{SND}. Then the \emph{FIO} $T_a^\varphi$ is bounded on $L^2(\R^n)$ if and only if $m=-n \, \max(0,(\delta-\rho)/2)$. In case $\rho\in [0,1],$ $\delta=1$ then the $L^2$-boundedness is valid if and only if $m<n(\rho-1)/2$.
\end{Prop}

A global result concerning the boundedness of FIOs  with amplitudes of order zero, which will be used in the proof of Theorem \ref{thm:sobolev}  goes as follows:
\begin{Lem}\label{sobolevtheorem}
Let $a (x,\xi)\in S_{1,0}^{0}(\R^n)$. Assume also that $\varphi(x,\xi) \in \Phi^2,$ is $\mathrm{SND}$. Then for $s\in \R$, the \emph{FIO} $T_a^\varphi$ is bounded from the Sobolev space $H^{s}(\R^n)$ to $H^{s}(\R^n)$.
\end{Lem}
\begin{proof}
This follows immediately from \cite[Theorem 5.7 part (ii)]{IRS}, by noting that the Besov-Lipschitz space $B^s_{p,q}(\R^n)$ in that result reduces to the Sobolev space $H^s(\R^n)$ when $p=q=2$.
\end{proof}

\noindent We also state the following version of the non-stationary phase lemma, whose proof can be found in \cite[Lemma 3.2]{RLS}.
\begin{Lem}\label{lem:non-stationary}
Let $\mathcal{K}\subset \mathbb{R}^n$ be a compact set and  $\Omega \supset \mathcal{K}$ an open set. Assume that $\Phi$ is a real valued function in $\mathcal C^{\infty}(\Omega )$ such that $|\nabla \Phi|>0$ and
$$|\d^\alpha \Phi| \lesssim |\nabla \Phi|,$$
for all multi-indices $\alpha$ with $|\alpha|\geq 1$.
Then, for any $F\in \mathcal C^\infty_c (\mathcal{K})$ and any integer $k\geq 0$,
\begin{equation*}
	 \Big| \int_{\R ^n} F(\xi)\, e^{i\Phi(\xi)}\ddd \xi \Big|
     \leq C_{k,n,\mathcal{K}} \sum_{|\alpha| \leq k} \int_\mathcal{K} |\d^{\alpha} F(\xi)| \, |\nabla \Phi(\xi)|^{-k}\ddd \xi.
\end{equation*}
\end{Lem}

\noindent Finally we recall a composition result, whose proof can be found in \cite[Theorem 4.2]{RRS}, or in a more general setting in \cite[Theorem 3.11]{CISY}, which will enable us to keep track of the parameter while a parameter-dependent $\Psi$DO acts from the left on a parameter-dependent FIO.
This will be crucial in the proof of the boundedness of FIOs with forbidden amplitudes on Sobolev spaces.

\begin{Prop}\label{thm:monsteriosity}
Let $m \leq 0$, $\displaystyle 0<\varepsilon <1/2$ and $\Omega := \R^n \times \{|\xi| > 1\}$. Suppose that $ a_t(x, \xi)\in S^m_{1,0} (\R^n)$ uniformly in $t \in (0, 1]$ and it is supported in $\Omega$, $b(\xi)\in S^0_{1,0}(\R^n)$ and $\varphi \in \mathcal C^\infty (\Omega)$ is such that
\begin{enumerate}
\item[\emph{(i)}]for constants $C_1, C_2 > 0$, $C_1|\xi| \leq |\nabla_x \varphi(x, \xi)| \leq C_2|\xi|$ for all $(x, \xi) \in \Omega$, and
\item[\emph{(ii)}]for all $|\alpha|, |\beta| \geq 1$, $|\partial_x^\alpha \varphi(x, \xi)|\lesssim  \la \xi \ra$ and $|\partial_\xi^\alpha \partial _x^\beta \varphi (x, \xi)| \lesssim  |\xi|^{1-|\alpha|}$, for all $(x, \xi) \in \Omega$.
\end{enumerate}

Consider the parameter dependent Fourier integral operator $T_{a_t}^\varphi$, given by \eqref{eq:FIO} with amplitude $a_t(x,\xi)$, and the parameter dependent Fourier multiplier
\begin{equation*}
b(tD)f(x) := \int_{\R^n} e^{ix\cdot \xi}\,b(t\xi)\,\widehat f(\xi) \ddd \xi.
\end{equation*}
Then the composition $b (tD)T_{a_t}^\varphi$ is also an \emph{FIO} with phase $\varphi$ and amplitude $\sigma_t$ which is given by
\begin{equation*}
\sigma_t(x, \xi) := \iint_{\R^{n}\times \R^n} a_t(y, \xi)\,b (t\eta)\,e^{i(x-y)\cdot \eta+i\varphi(y,\xi)-i\varphi(x,\xi)} \ddd\eta \dd y.
\end{equation*}
Moreover, for each $M \geq 1$, we can write $\sigma_t$ as
\begin{equation*}
\sigma_t(x, \xi) = b (t \nabla_x \varphi(x, \xi))\,a_t(x, \xi) + \sum_{0<|\alpha|<M} \frac{t^{|\alpha|}}{\alpha!}
 \sigma_\alpha (t, x, \xi) + t^{M\varepsilon} r(t, x, \xi),
\end{equation*}
for $t \in (0, 1)$. Moreover, for all multi-indices $\beta, \gamma$ one has
\begin{equation*}
\sup_{t\in(0,1)} \abs{\partial ^\gamma_\xi \partial_x^\beta \sigma_\alpha(t, x, \xi)   t^{|\alpha|(1-\varepsilon)}} \lesssim \la\xi\ra^{m-|\alpha|\left ( 1/2- \varepsilon \right )-|\gamma|}\text{ for } 0 < |\alpha | < M,
\end{equation*}
and
\begin{equation*}
\sup_{t\in(0,1)} \abs{\partial _\xi^\gamma \partial_x^\beta r(t, x, \xi)} \lesssim   \la\xi\ra ^{m-M\left ( 1/2 -\varepsilon \right )-|\gamma|}.
\end{equation*}
\end{Prop}

\begin{Rem}
The composition \emph{Proposition \ref{thm:monsteriosity}} as well as the forthcoming \emph{Theorem \ref{thm: main composition} }are both formulated for \emph{FIOs} where the corresponding amplitudes are assumed to vanish in a neighbourhood of the origin. This is just to avoid the singularity of the phase function at the origin. However, as we shall see in \emph{Remark \ref{low freq rem}}, this won't cause any problems for the validity of our results.
\end{Rem}

\section{Reduction and globalisation }\label{RS globalisation}

We start by describing how the problem of $L^p$-boundedness of FIOs with SND phase functions that belong to the class $\Phi^2$, can be reduced to the case of operators that are well-suited for the Ruzhansky-Sugimoto's globalisation procedure.\\


Thus let $T_a^{\varphi}$ be an FIO given by \eqref{eq:FIO}, with $\varphi\in \Phi^2$ and SND. We start by localising the amplitude in the $\xi$ variable by introducing an open convex covering $\{U_l\}_{l=1}^{M},$ with maximum of diameters $d$, of the unit sphere $\mathbb{S}^{n-1}$.
Let $\Xi_{l}$ be a smooth partition of unity subordinate to the covering $U_l$ and set $$a_{l}(x,\xi):=a(x,\xi)\, \Xi_{l}\brkt{\frac{\xi}{|\xi|}}.$$
Define
\begin{equation*}
  T_{l}f(x):=  \int_{\R^n} \, a_l(x,\xi)\, e^{i\varphi(x,\xi)}\,\widehat{f}(\xi) \, \ddd\xi,
\end{equation*}
and fix a point $\zeta_l \in U_l.$  Then for any $\xi\in U_l$, Taylor's formula and Euler's homogeneity formula yield
\begin{align*}
  \varphi(x,\xi) &= \varphi(x,\zeta_l) +  \nabla_{\xi}\varphi (x,\zeta_l)\cdot (\xi-\zeta_l) +\lambda (x,\xi) \\
  &= \lambda(x,\xi)+ \nabla_{\xi}\varphi(x,\zeta_l)\cdot\xi.
\end{align*}
Furthermore, for $\xi\in U_l$,
$$\partial_{\xi_k} \lambda(x,\xi)= \partial_{\xi_k} \varphi\Big(x,\frac{\xi}{|\xi|}\Big)-\partial_{\xi_k} \varphi(x,\zeta_l),$$
so the mean-value theorem and the definition of class $\Phi^2$ yield $$|\partial_{\xi_k}\partial^{\beta}_{x} \lambda(x,\xi)|\leq Cd,$$
for all $|\beta|\geq 0$, and for $|\alpha|\geq 2$, $$|\partial^{\alpha}_{\xi}\partial^{\beta}_{x} \lambda(x,\xi)|\leq C |\xi|^{1-|\alpha|},$$
for all $\beta$. One also observes that due to the homogeneity of $\lambda(x,\xi)$ in $\xi$ and the mean-value theorem, one also has that $|\nabla_x\lambda(x, \xi)|\lesssim |\xi|$. We shall now extend the function $\lambda(x,\xi)$ to the whole of $\mathbb{R}^{n}\times \mathbb{R}^{n}\setminus \{ 0 \}$, preserving its properties and we denote this extension by $\lambda(x,\xi)$ again. Now this $\lambda$ belongs to the class $\Phi^1$. Hence the Fourier integral operators $T_l$ defined by
\begin{equation*}
  T_{l}f(x):= \int_{\R^n} a_l(x,\xi)\,e^{i\lambda(x,\xi)+i  \nabla_{\xi}\varphi(x,\zeta_l)\cdot\xi}\, \widehat{f}(\xi) \, \ddd\xi,
\end{equation*}
are the localised pieces of the original Fourier integral operator $T$ and therefore $$T=\sum_{l=1}^{M} T_l.$$
Now, let us investigate the $L^p$-boundedness of each piece $T_l$. To this end we observe that due to the SND assumption on $\varphi$, the map $\mathbf t_{l}(x):= \nabla_\xi \varphi(x,\zeta_l)$ is a global diffeomorphism and composing $T_l f(x)$ with the inverse of $\mathbf t_{l}$ results in the FIO
\begin{equation*}
  (T_{l}f)(\mathbf t^{-1}_{l}(x))= \int_{\R^n} a_l(\mathbf t^{-1}_{l}(x),\xi)\,e^{i\lambda(\mathbf t^{-1}_{l}(x),\xi)+ix\cdot \xi}\, \widehat{f}(\xi) \, \ddd\xi.
\end{equation*}
Observe that all the derivatives of $\mathbf t^{-1}_{l}$ are bounded and indeed the phase function $\lambda(\mathbf t^{-1}_{l}(x),\xi)+x\cdot \xi$ is SND (note also that the diameters $d$ can be picked as small as we like).
Therefore the study of the global $L^p$-boundedness of $T_l$ is reduced to the study of the global $L^p$-boundedness of FIOs of the form
\begin{equation*}
   \int_{\R^n} {\sigma(x,\xi)}\,e^{i\theta(x,\xi)+ix\cdot \xi}\, \widehat{f}(\xi) \, \ddd\xi,
\end{equation*}
where $\sigma(x,\xi)$ belongs to the same amplitude-class as $a(x,\xi)$, {$\theta\in \Phi^1$} and $\theta(x,\xi)+ x\cdot \xi$ being SND.\\

In a similar way, one can show that, for an FIO of the form $$ \iint_{\R^n\times \R^n} a(y,\xi)\, e^{i\varphi(y,\xi)-ix\cdot\xi}\, f(y) \ddd \xi \dd y,$$with $\varphi\in \Phi^2$, matters can be reduced to FIOs of the form

\m{
   \iint_{\R^n\times \R^n} {\sigma(y,\xi)}\, e^{i\theta(y,\xi)+i(y-x)\cdot\xi}\, f(y)  \ddd\xi \dd y,
}
where $\sigma(y,\xi)$ belongs to the same class as $a(y,\xi)$ and $\theta(y,\xi)\in \Phi^{1}.$\\

Now with these reductions in mind we proceed to describe the globalisation procedure. In \cite{RuSu}, Ruzhansky and Sugimoto developed a new technique to transfer local boundedness of Fourier integral operators, which was proven by Seeger, Sogge and Stein \cite{SSS}, to a global result, where the amplitudes of the corresponding operators do not have compact spatial supports. We also note that this transference method only works for $\rho\neq 0$. In order to prove global regularity results we follow \cite{RuSu}, and first assume that the phase function is smooth in the the support of the amplitude (this assumption can be removed by dividing the operator into low and high frequency portions and treat each one separately). Then one defines the function
$$H(x,y,z):=\inf_{\xi\in\R^n}\abs{z+\nabla_\xi\vartheta(x,y,\xi)},
$$
where for us $\vartheta(x,y,\xi)$ is either $\theta(x,\xi)$ or $-\theta(y,\xi)$, with $\theta\in \Phi^1$ and
\[
\Delta_r:=\{(x,y,z)\in\R^n\times\R^n\times\R^n:H(x,y,z)\geq r\}.
\]
One also defines
$$
\widetilde H(z)
:=\inf_{x,y\in\R^n}H(x,y,z)=
\inf_{x,y,\xi\in\R^n}\abs{z+\nabla_\xi\vartheta(x,y,\xi)}
$$
and
$$
\widetilde\Delta_r
:=\set{z\in\R^n:\widetilde H(z)\geq r}.
$$
Let $\sigma(x, y,\xi)\in \mathcal{C}^{\infty}(\Rn \times \Rn \times \Rn)$ satisfying the estimate
\m{
\left |\partial_\xi^\alpha \partial_x^\beta \partial_y^\gamma \sigma(x, y,\xi) \right |\lesssim \langle\xi\rangle ^{m-\rho|\alpha|+\delta|\beta +\gamma|},
}
with $m\leq 0$, for all multi-indices $\alpha$, $\beta$ and $\gamma$ and $(x,y,\xi)\in \R ^n\times \R ^n \times \Rn$.
We set
$$
M_L :=\sum_{|\gamma|\leq L}\sup_{x,y,\xi\in\R^n}
\abs{\langle\xi\rangle^{-(m-\rho|\gamma|)}\,\partial^\gamma_\xi \sigma(x,y,\xi)}
$$
and
$$
N_L :=\sum_{1\leq |\gamma|\leq L}\sup_{x,y,\xi\in\R^n}
\abs{\langle\xi\rangle^{-(1-|\gamma|)}\,\partial^\gamma_\xi \vartheta(x,y,\xi)}.
$$
Here we observe that $N_L <\infty$ by the $\Phi^1$-condition on $\theta$ above. Given these definitions one has the following lemma.
\begin{Lem}\label{Lem:outside}
Let $r\geq1$ and $L\geq 1$.
Then we have
$\R^n \setminus \widetilde{\Delta}_{2r} \subset \{z:\, |z|<(2+N_L )r \}$. Furthermore for $r>0$, $x\in\widetilde{\Delta}_{2r}$ and $|y|\leq r$
we have
\begin{equation}\label{H bounded by H}
 \widetilde{H}(x)\leq 2H(x,y, x-y)
\end{equation}
and therefore $(x,y,x-y)\in\Delta_r$.
\end{Lem}
\begin{proof}
For $z\in\R^n\setminus\widetilde{\Delta}_{2r}$,
we have
$\widetilde{H}(z)
<2r$.
Hence, there exist $x_0,y_0,\xi_0\in\R^n$ such that
$$
|z+\nabla_\xi\vartheta(x_0,y_0,\xi_0)|<2r.
$$
Since, $r\geq 1$, this yields that
\[
|z|\leq|z +\nabla_\xi\vartheta(x_0,y_0,\xi_0)|+|\nabla_\xi\vartheta(x_0,y_0,\xi_0)|
\leq 2r+N_L \leq (2+N_L )r.
\]

The claim that $(x,y,x-y)\in\Delta_r$ follows from \eqref{H bounded by H} and the definition of $ \Delta_r$. Therefore it only remains to prove \eqref{H bounded by H}. Now, if $|y|\leq r$ and $x\in \widetilde{\Delta}_{2r}$ then since $  \widetilde{H}(x)\geq 2r$, we have that
\begin{align*}
 \widetilde{H}(x)
&\leq |x+\nabla\vartheta(x,y,\xi)|
\leq|x-y+\nabla\vartheta(x,y,\xi)|+|y|
\\
&\leq|x-y+\nabla\vartheta(x,y,\xi)|+  \frac{\widetilde{H}(x)}2.
\end{align*}

From this, \eqref{H bounded by H} follows at once.
\end{proof}
In order to prove the global boundedness, the following result is of particular importance.
\begin{Lem}\label{Lem:kernel RuzhSug}
The kernel
$$K(x,y,z)
:=\int_{\R^n}e^{iz\cdot\xi+i\vartheta(x,y,\xi)}\, \sigma (x,y,\xi)\ddd\xi$$
is smooth on $\cup_{r>0}\Delta_r$. Moreover, for all $L>  n/\rho$ and $r\geq 1$ it satisfies
\begin{equation}\label{boundedness of HLK}
\|H^{L} K\|_{L^\infty(\Delta_r)}\leq C( L,M_L ,N_{L+1}),
\end{equation}
where $C( L, M_L ,N_{L+1})$ is a positive constant depending
only on $L$, $M_L$ and $N_{L+1}$.
For $ L> n$ and $r\geq 1$, the function $\widetilde{H}(z)$ satisfies the bound
\begin{equation}\label{integrable HL}
\| \widetilde{H}^{-L}\|_{L^1(\widetilde{\Delta}_r)}
\leq C(L,N_{L+1}).
\end{equation}
\end{Lem}
\begin{proof}
If one introduces the differential operator
$$
D:=\frac{(z+\nabla_\xi\vartheta)\cdot\nabla_{\xi}}{i|z+\nabla_\xi\vartheta|^2},
$$
with the transpose $D^*$, then integrating by parts $L$ times yields
$$
K(x,y,z)=\int_{\R^n}e^{iz\cdot\xi+i\vartheta(x,y,\xi)}
\left(D^*\right)^{L} \sigma (x,y,\xi)\ddd\xi.
$$

Now \eqref{boundedness of HLK} follows from the relation
$$r\leq H(x,y,z)\leq|z+\nabla_\xi\vartheta(x,y,\xi)|,$$
which \smallskip is valid for $(x,y,z)\in\Delta_r$ and $\xi\in\R^n$.
Moreover
$$|z|\leq |z+\nabla_\xi\vartheta(x,y,\xi)|+N_{L+1},$$
for any $\xi\not=0$, which yields that
$$|z|\leq  \widetilde{H}(z)+N_{L+1}.$$
Hence for $|z|\geq 2N_{L+1}$ one has
$$ |z|\leq \widetilde{H}(z)+ |z|/2$$
and therefore
$$|z|\leq 2 \widetilde{H}(z).$$
Using this we get
\begin{align*}
\| \widetilde{H}^{-L}\|_{L^1(\widetilde{\Delta}_r)}
&\leq
\| \widetilde{H}^{-L}\|_{L^1(\widetilde{\Delta}_r\cap\{|z|\leq 2N_{L+1}\})}+
\| \widetilde{H}^{-L}\|_{L^1(\widetilde{\Delta}_r\cap\{|z|\geq 2N_{L+1}\})}
\\
&\leq
r^{-L}\int_{|z|\leq 2N_{L+1}}  \dd z +
2^{L}\int_{|z|\geq 2N_{L+1}}|z|^{-L}\dd z \\ & \leq C( L, N_{L+1}),
\end{align*}
which proves \eqref{integrable HL} .
\end{proof}\hspace*{1cm}

In Section \ref{Lp section} in the proof of Theorem \ref{thm:LpResult} (Step 3 of the proof), we shall see how Lemmas \ref{Lem:outside} and \ref{Lem:kernel RuzhSug} are used to globalise the local $L^p$-boundedness result.

\section{{Composition of Fourier multipliers and FIOs}}\label{Composition FIOs}

We start with a composition theorem which allows us to
left-compose a Fourier multiplier with an FIO. The difference between Theorem \ref{thm: main composition} below and Proposition \ref{thm:monsteriosity} lies in the fact that, although the latter  deals with the parameter dependent case, it only covers amplitudes in $S^m_{1,0}(\R^n)$. Also the method of proof of Proposition \ref{thm:monsteriosity} is quite different from that of the following theorem whose proof is not just a modification of the former. The difficulties arise exactly when $\delta\geq \rho,$ but they can be overcome. However, as we shall see, the forbidden case of $\delta=1$ has to be excluded.\\

Note that the composition theorem below also allows us to compose our more general FIOs with Bessel potential operators (see i.e. Lemma \ref{lemma:HLS-lemma}) in order to obtain crucial $\mathscr{H}^q-L^2$ estimates that are in turn used in the proofs of the $L^p$-boundedness results (Theorem \ref{thm:LpResult}).

\begin{Th}\label{thm: main composition}
Let $m, m'\in\R$, $\rho \in [0,1], \delta \in [0,1)$  and $\Omega := \R^n \times {\{|\xi| > 1\}}$.
Suppose that $ a(x, \xi)\in S_{\rho,\delta}^m(\Rn) $, and it is supported in $\Omega$, $\gamma(\xi)\in S^{m'}_{1,0}(\Rn)$ and $\varphi \in \mathcal C^\infty (\Omega)$ is such that
\begin{enumerate}
\item[$(i)$]for constants $C_1, C_2 > 0$, $C_1|\xi| \leq |\nabla_x \varphi(x, \xi)| \leq C_2|\xi|$ for all $(x, \xi) \in \Omega$, and
\item[$(ii)$]for all $|\alpha|, |\beta| \geq 1$, $|\partial_x^\alpha \varphi(x, \xi)|\lesssim  \la \xi \ra$ and $|\partial_\xi^\alpha \partial _x^\beta \varphi (x, \xi)| \lesssim  1$, for all $(x, \xi) \in \Omega$.
\end{enumerate}

\noindent Consider the Fourier multiplier and the Fourier integral operator
\begin{equation*}
\gamma(D)f(x) := \int_{\R^n} e^{ix\cdot \xi}\,\gamma(\xi)\,\widehat f(\xi) \ddd \xi
\quad \text{\emph{ and }} \quad
T_{a}^\varphi f(x) := \int_{\R^n} e^{i\varphi(x,\xi)}\,a(x, \xi)\,\widehat f(\xi) \ddd\xi.
\end{equation*}
Then the composition operator $T_{b}^\varphi:=\gamma (D)T_{a}^\varphi $ is also an \emph{FIO} $($with the same phase as $T^\varphi_a$$)$, and with amplitude given by
\begin{equation}\label{eq:compampl}
b(x, \xi) := \iint_{\R^{n}\times \R^n} a(y, \xi)\,\gamma (\eta)\,e^{i(x-y)\cdot \eta+i\varphi(y,\xi)-i\varphi(x,\xi)} \ddd\eta \dd y.
\end{equation}
Moreover $b\in S^{m+m'}_{\rho, \delta}(\Rn)$.
\end{Th}

\begin{Rem}
It is easy to show that if a phase function $\varphi\in \Phi^2$ is \emph{SND} then it satisfies all the requirements of \emph{Theorem \ref{thm: main composition}}.
\end{Rem}

\begin{proof}[Proof of \emph{Theorem \ref{thm: main composition}}]

The expression in \eqref{eq:compampl} can easily be derived through a simple calculation. Set
$$c(x,\xi,\eta):=a(x,\xi)\,\gamma(\eta)$$
and
$$\Phi(x,y,\xi,\eta):=(x-y)\cdot \eta+\varphi(y,\xi)-\varphi(x,\xi).$$
Then the following estimates are valid:
\begin{equation}\label{eq:forrelations}
\begin{split}
\langle \nabla_{\eta} \Phi\rangle&= \langle x-y\rangle,\\
\langle \nabla_{y} \Phi \rangle&\geq C_R \langle \xi-\eta\rangle,\,\,\, |\xi|\leq R, \, \, R>0,\\
|\nabla _x \Phi|&\lesssim |\xi| +|\xi-\eta| ,\\
|\nabla_{\xi} \Phi|&\lesssim |x-y|.
\end{split}
\end{equation}
Indeed the first equality is trivial, and for the second one, setting
\begin{equation*}
    \Lambda(x,\xi):=\varphi(x,\xi)-x\cdot \xi,
\end{equation*}
 we note that for $|\xi|<R$ and $R\geq 1$, triangle inequality and condition $(i)$ on the phase yield that
\begin{equation*}
    \begin{split}
2(C_2 +1)R (1+ |\nabla_{y} \Phi|)
& \geq 2(C_2 +1)R + |\nabla_{y} \Phi|\\
& =  2(C_2 +1)R +| \nabla_{y} (\varphi(y,\xi)-y\cdot\eta)|\\
& = 2(C_2 +1)R + | \nabla_{y} (y\cdot(\xi- \eta)+\Lambda(y,\xi))|\\
& = 2(C_2 +1)R +|\xi-\eta+ \nabla_{y} \Lambda(y,\xi)|\\
& \geq 2(C_2 +1)R + |\xi-\eta|-(C_2+1)  |\xi|\\
& \geq 2(C_2 +1)R -(C_2 +1)R  + |\xi-\eta|\\
& =(C_2 +1)R + |\xi-\eta| \\
& \geq (1+|\xi-\eta|).
\end{split}
\end{equation*}

On the other hand, for $|\xi|<R$ and $R<1$, condition $(i)$ on the phase implies
\begin{equation*}
    \begin{split}
(C_2+1)(1+ |\nabla_{y} \Phi|)
& \geq (C_2+1)+ \frac{1}{2}|\nabla_{y} \Phi|\\
& = (C_2+1)+\frac{1}{2}|\xi-\eta+ \nabla_{y} \Lambda(y,\xi)|\\
& \geq (C_2+1)+ \frac{1}{2}|\xi-\eta|-\frac{C_2+1}{2} |\xi|\\
& \geq (C_2+1)-\frac{R(C_2+1)}{2} + \frac{1}{2}|\xi-\eta|\\
& \geq \frac{1}{2}(1+ |\xi-\eta|).
\end{split}
\end{equation*}
To show the third estimate in \eqref{eq:forrelations} we observe that
condition $(i)$ on the phase
yields
\begin{equation*}
    \begin{split}
|\nabla_{x} \Phi|=| \nabla_{x} (x\cdot(\eta-\xi)-\Lambda(x,\xi))|\lesssim |\xi-\eta|+ (C_2+1)|\xi|\lesssim  |\xi-\eta|+ |\xi|.
\end{split}
\end{equation*}
Finally to show the forth estimate in \eqref{eq:forrelations} we observe that the mean-value theorem and condition $(ii)$ on the phase yield that
\begin{equation*}
    \begin{split}
|\nabla_{\xi} \Phi|=| \nabla_{\xi} \varphi(y,\xi)-\nabla_{\xi}\varphi(x,\xi)|\lesssim |x-y|.
\end{split}
\end{equation*}
Using Fa\`a di Bruno's formulae and estimate \eqref{eq:forrelations} we can also show that
\begin{equation}\label{estimate for the exp of the phase}
|\partial_x^{\beta}\partial_\xi^{\alpha}e^{i\Phi}|\lesssim \langle x-y\rangle^{|\alpha|}(1+|\xi|^2+ |\xi-\eta|^2) ^{|\beta|/2}.
\end{equation}

Introduce the differential operators
$$L_\eta:=\langle \nabla_{\eta}\Phi\rangle^{-2}\,(1-i\nabla_\eta\Phi\cdot\nabla_\eta)$$
$$L_y:=\langle \nabla_{y}\Phi\rangle^{-2}\,(1-i\nabla_y\Phi\cdot\nabla_y),$$ and integrating by parts we have
\begin{equation*}
b(x, \xi) = \iint_{\R^{n}\times \R^n} e^{i\Phi}\, (L_\eta^*)^{N}(L_y^*)^{N}c(y, \xi,\eta)\ddd\eta\dd y,
\end{equation*}
for large positive $N$. It follow from \eqref{eq:forrelations} and \eqref{estimate for the exp of the phase} that, for any $R>0$, \linebreak$b(x,\xi)\in \mathcal{C}^{\infty}_{b}(\R_{x}^n \times B(0,R))$ (the subscript $b$ indicates the boundedness of all the derivatives). It also follows from condition $(i)$ on the phase that
\begin{equation}\label{yderivative of the phase}
|\nabla_y \Phi|= |\nabla_y \varphi(y, \xi)-\eta|\geq C_1|\xi|-|\eta|.
\end{equation}
Now let $\chi\in \mathcal{C}_c^{\infty}(\R^n)$ be such that $0\leq \chi(x)\leq 1$ and $\chi(x)=1$ when $|x|\leq 1/2$ and $\chi(x)=0$ when $|x|\geq 2/3.$
Set
$$\chi_1(\eta,\xi):=\chi\Big (\frac{\eta}{C_1\langle\xi\rangle}\Big).$$
Since $|\eta|\leq 2 C_1\langle \xi\rangle/3$ on the support of $\chi_1$, \eqref{yderivative of the phase} yields (on $\supp\chi_1$)
$$|\nabla_y\Phi|\geq C_1\Big(|\xi|-\frac{2}{3}\langle \xi\rangle\Big).$$

At this point , we observe that since $$\lim_{|\xi|\to \infty} C_1\Big(\frac{|\xi|}{\langle \xi\rangle}-\frac{2}{3}\Big)=\frac{C_1}{3},$$
it follows that there exists $R_1>0$ and $C>0$ such that on $\supp \chi_1 \cap \{|\xi|\geq R_1\}$ one has
\begin{equation}\label{vetefan 2}
|\nabla_y\Phi|\geq C \langle \xi\rangle= C\Big(\frac{1}{2} +\frac{|\xi|^2}{2} + \frac{\langle \xi\rangle^2}{2}\Big) ^{1/2} \gtrsim (1 +|\xi|^2 + |\eta|^2) ^{1/2} .
\end{equation}
Setting
\begin{equation*}
b_1(x,\xi):=\iint_{\R^{n}\times \R^n} c(y, \xi,\eta)\,\chi_1 (\eta,\xi)\,e^{i\Phi} \ddd\eta \dd y,
\end{equation*}
and
$$\tilde L_y:= -i |\nabla_y\Phi|^{-2} \nabla_y \Phi \cdot \nabla_y,$$
and integrating by parts yields
\begin{equation*}
b_1(x,\xi)= \iint_{\R^{n}\times \R^n} e^{i\Phi}\,(L_\eta^*)^{N}(\tilde L_y^*)^{N}(\chi_1(\eta,\xi)\,c(y, \xi,\eta))\ddd\eta\dd y.
\end{equation*}
Since $b(x,\xi)\in \mathcal{C}^{\infty}_{b}(\R_{x}^n \times B(0,R))$, estimate \eqref{vetefan 2} yields that $b_1(x,\xi)\in S^{-\infty}(\Rn)$.\\

 Now define $\chi_2: =1-\chi_1$ and consider
 \begin{equation*}
b_2(x,\xi):=\iint_{\R^{n}\times \R^n} c(y, \xi,\eta)\,\chi_2 (\eta,\xi)e^{i\Phi} \ddd\eta \dd y.
\end{equation*}
To simplify the calculations from here we set
$$I(x,y,\xi):= \int_{0}^{1}\nabla_x \Lambda(y+s(x-y), \xi) \dd s.$$
Noting that $\nabla_x \varphi(x,\xi)=\nabla_x \Lambda(x,\xi)+\xi$, rewriting
$$\Phi(x,y,\xi,\eta)
= (x-y)\cdot (\eta-\xi-I(x,y,\xi)),$$
making the change of variables
$$z:=y-x, \quad
\zeta:=\eta-\xi-I(x,y,\xi),$$
and then defining
 $$c_1(y,\xi,\eta):= \chi_2 (\eta+I(x,y,\xi), \xi)\, c(y,\xi, \eta+I(x,y,\xi)),$$
 we obtain
\begin{equation*}
b_2(x,\xi)=\iint_{\R^{n}\times \R^n} e^{-iz\cdot\zeta}\,c_1(x+z, \xi,\xi+\zeta)\ddd\zeta \dd z.
\end{equation*}
Moreover, $b_2(x,\xi)$ in turn can be split into $b_3(x,\xi)+ b_4(x,\xi)$ with

\begin{equation*}
b_3(x,\xi):= \iint_{\R^{n}\times \R^n} e^{-iz\cdot\zeta}\chi_2(\zeta, \xi)\,c_1(x+z, \xi,\xi+\zeta)\ddd\zeta \dd z
\end{equation*}
and
\begin{equation*}
b_4(x,\xi):= \iint_{\R^{n}\times \R^n} e^{-iz\cdot\zeta}\chi_1(\zeta, \xi)\,c_1(x+z, \xi,\xi+\zeta)\ddd\zeta \dd z.
\end{equation*}
We observe that on the support of $\chi_2(\zeta, \xi)$ one has that
\begin{equation*}
|\zeta|\geq\frac{1}{2}C_1 \langle \xi\rangle
\end{equation*}
and on the support of $\chi_2 (I(x,y,\xi)+ \xi+\zeta, \xi)$,
\begin{equation*}
|I(x,y,\xi)+ \xi+\zeta |\geq \frac{1}{2}C_1 \langle \xi\rangle.
\end{equation*}
Therefore integrating by parts we can show that
\m{
&b_3(x,\xi)=  \iint_{\R^{n}\times \R^n} e^{-iz\cdot\zeta}|\zeta|^{-2N}\\
&\qquad \qquad  \qquad
\times(-\Delta_z)^{N}\{\langle z\rangle^{-2N}(1-\Delta_{\zeta})^{N}(\chi_2(\zeta, \xi)\,c_1(x+z, \xi,\xi+\zeta))\}\ddd\zeta \dd z,
}
to conclude that $b_3(x,\xi)\in S^{-\infty}(\Rn).$\\

Defining
$$c_2(y,\xi,\eta):= \chi_1(\eta-\xi,\xi)\,c_1(y,\xi,\eta),$$
we see that
\begin{equation*}
\begin{split}
    c_2(y,\xi,\eta)= \chi_1(\eta-\xi,\xi) \chi_2 (\eta+I(x,y,\xi), \xi)\, c(y,\xi,\eta+I(x,y,\xi)),
\end{split}
\end{equation*}
and $b_4(x,\xi)$ can be written as
\begin{equation}\label{defn of b4}
b_4(x,\xi)=\iint_{\R^{n}\times \R^n} e^{-iz\cdot\zeta}\,c_2(x+z, \xi,\xi+\zeta)\ddd\zeta \dd z.
\end{equation}
Using the definition of $c_2(y,\xi,\eta)$ we see that
\begin{equation*}
\begin{split}
&|\eta-\xi|\leq \frac{2}{3} C_1  \langle \xi\rangle\\
&|\eta+I(x,y,\xi) |\geq \frac{1}{2} C_1 \langle \xi\rangle,
\end{split}
\end{equation*}
on the support of $c_2(y,\xi,\eta)$.\\

In what follows we shall denote the derivative of $c_2(y,\xi,\eta)$ with respect to $y$ by $\partial_1 c_2$, the derivative of $c_2(y,\xi,\eta)$ w.r.t.  $\xi$ by $\partial_2 c_2$, and the derivative of $c_2(y,\xi,\eta)$ w.r.t. $\eta$ by $\partial_3 c_2$. Now using \eqref{defn of b4} and Taylor's formula we have that
\begin{equation*}
\begin{split}
   &b_4(x,\xi)=
   \sum_{|\nu|<N} \frac{(-i)^{|\nu|}}{\nu!}\iint_{\R^{n}\times \R^n} e^{-iz\cdot\zeta}\, (\partial^{\nu}_1 \partial^{\nu}_3  \,c_2) (x+z, \xi,\xi)\ddd\zeta \dd z\\&\quad +N\sum_{|\nu|=N} \int_{0}^{1} \frac{(1-s)^{N-1}}{\nu!} \iint_{\R^{n}\times \R^n}  e^{-iz\cdot\zeta}\, (\partial^{\nu}_1 \partial^{\nu}_3\,c_2) (x+z, \xi,\xi+s\zeta) \ddd\zeta \dd z \dd s.
  \end{split}
\end{equation*}

For $s\in [0,1]$, set
$$\sigma (x+z, \xi,\xi+s\zeta):= (\partial^{\nu}_1 \partial^{\nu}_3 \,c_2) (x+z, \xi,\xi+s\zeta).$$
Let us now study the behaviour of the derivatives of
\begin{equation*}
 r_s(x,\xi):=   \iint_{\R^{n}\times \R^n} e^{-iz\cdot\zeta}\,\sigma (x+z, \xi,\xi+s\zeta)\ddd\zeta \dd z.
\end{equation*}
 To this end observe that
\begin{equation*}
\partial^{\alpha}_{\xi}\partial^{\beta}_{x}  r_s(x,\xi)=\iint_{\R^{n}\times \R^n} e^{-iz\cdot\zeta}\, \partial^{\alpha}_{\xi}\partial^{\beta}_{x}\sigma (x+z, \xi,\xi+s\zeta)\ddd\zeta \dd z.
\end{equation*}

Now for $M>n/2$  we write $$e^{iz\cdot\zeta}=(1+\langle\xi\rangle^{2\delta} |z|^2)^{-M}(1+\langle\xi\rangle^{2\delta}(-\Delta_\zeta))^{M}e^{iz\cdot\zeta}$$
and integration by parts yields that
\begin{equation}\label{derivatives of b4}
    \partial^{\alpha}_{\xi}\partial^{\beta}_{x}  r_s(x,\xi)=\iint_{\R^{n}\times \R^n} e^{-iz\cdot\zeta}\,  R_s(x,\xi, z,\zeta)\ddd\zeta \dd z,
\end{equation}
with
\begin{equation}\label{defn: def av r}
    R_s (x,\xi, z,\zeta):= (1+\langle\xi\rangle^{2\delta} |z|^2)^{-M}(1+\langle\xi\rangle^{2\delta}(-\Delta_\zeta))^{M}\partial^{\alpha}_{\xi}\partial^{\beta}_{x}\sigma(x+z, \xi,\xi+s\zeta).
\end{equation}
We also have that
\begin{equation}\label{useful derivatives}
\begin{split}
    &|(1+\langle\xi\rangle^{2\delta}(-\Delta_\zeta))^{M}\,\partial^{\alpha}_{\xi}\partial^{\beta}_{x}\partial^{\gamma}_{z}\sigma(x+z, \xi,\xi+s\zeta)|\\&\qquad\lesssim  \sum_{|\lambda|<M}\langle\xi\rangle^{2\delta |\lambda|}\,(\partial^{\beta+\gamma}_{1}\partial^{\alpha}_{2}\partial^{2\lambda}_{3}\,c_2)(x+z, \xi,\xi+s\zeta) \\& \qquad\qquad\qquad+\sum_{|\lambda|<M}\langle\xi\rangle^{2\delta |\lambda|}\,(\partial^{\beta+\gamma}_{1}\partial^{2\lambda+\alpha}_{3}\,c_2)(x+z, \xi,\xi+s\zeta)\\&\qquad\lesssim  \sum_{|\lambda|<M}\langle\xi\rangle^{2\delta |\lambda|}\, \langle\xi\rangle^{m+\delta|\beta|+ \delta|\gamma|-\rho|\alpha|}\,\langle\xi+s\zeta\rangle^{m'-2|\lambda|}\\&  \qquad\qquad\qquad+\sum_{|\lambda|<M}\langle\xi\rangle^{2\delta |\lambda|}\, \langle\xi\rangle^{m+\delta|\beta|+ \delta|\gamma|}\,\langle\xi+s\zeta\rangle^{m'-2|\lambda|-|\alpha|}.
\end{split}
\end{equation}
Divide the domain of integration in $\zeta$ in \eqref{derivatives of b4} into three pieces
$\bold{A}:=\{|\zeta|\leq \langle\xi\rangle^\delta /2\}$, $\bold{B}:=\{\langle\xi\rangle^\delta/2\leq |\zeta|\leq \langle\xi\rangle/2\}$ and
$\bold{C}:=\{|\zeta|\geq \langle\xi\rangle/2\}.$\\

We observe that for $s\in [0,1]$,
$|\zeta|\leq  \langle\xi\rangle / 2$ we get
$$|\langle \xi+s\zeta\rangle -\langle \xi\rangle|=\Big|s\int_{0}^{1} \sum_{k=1}^n\zeta_j \partial_{\xi_j} \langle \xi+ts\zeta\rangle\dd t\Big|{\leq} \frac{\langle \xi\rangle}{2},$$
since
$|\partial_{\xi_j} \langle \xi+ts\zeta\rangle |\leq 1$.
This implies that for $s\in [0,1]$, $|\zeta|\leq \langle\xi\rangle/2$ we have
\begin{equation}\label{mixed japanese estim}
     \frac{\langle\xi\rangle}{2}\leq  \langle\xi+s\zeta\rangle \leq \frac{3\langle\xi\rangle}{2}.
\end{equation}

Moreover \eqref{useful derivatives} and \eqref{mixed japanese estim} also yield that in in $\bold{A}\cup \bold{B}$ we have
\begin{equation*}
    | R_s(x,\xi, z,\zeta)|\lesssim (1+\langle\xi\rangle^{2\delta} |z|^2)^{-M}\jap\xi^{m+m'-\rho|\alpha|+\delta|\beta|}.
\end{equation*}
Hence
\m{\Big|\int_{\bold{A}} \int_{\R^n}e^{-i z\cdot\zeta}\,  R_s \dd z \ddd \zeta \Big|&\lesssim \jap\xi ^{m+m'-\rho|\alpha|+\delta|\beta|-\delta n}\int_\bold{A} \int_{\R^n}(1+|u|^2)^{-M}  \dd u\ddd \zeta \\&\lesssim \jap\xi ^{m+m' -\rho|\alpha|+\delta|\beta|}.}
Next we observe that since
\begin{equation}\label{derivative of japs}
    |\partial^\gamma_z((1+\langle\xi\rangle^{2\delta} |z|^2)^{-M})|\lesssim \jap\xi^{\delta |\gamma|} (1+\langle\xi\rangle^{2\delta} |z|^2)^{-M},
\end{equation}
one has
$$|(-\Delta_z)^{M} R_s|\lesssim \jap\xi^{m+m'-\rho|\alpha|+\delta|\beta|+2M\delta}(1+\langle\xi\rangle^{2\delta} |z|^2)^{-M}.$$ Therefore
\begin{equation*}
\begin{split}
     &\Big|\int_{\bold{B}} \int_{\R^n}e^{-i z\cdot\zeta} \, R_s \dd z\ddd \zeta \Big| = \Big|\int_{\bold{B}} \int_{\R^n}e^{-i z\cdot\zeta} |\zeta|^{-2M} \, (-\Delta_z)^M {R_s} \dd z \ddd \zeta \Big|\\
     &\qquad\lesssim \jap\xi ^{m+m'-\rho|\alpha|+\delta|\beta|+(2M-n)\delta }\int_{\bold{B}} |\zeta|^{-2M}  \int_{\R^n}(1+|u|^2)^{-M}  \dd u\ddd \zeta \\&\qquad\lesssim \jap\xi ^{m+m'-\rho|\alpha|+\delta|\beta|}.
\end{split}
\end{equation*}

On the set $\bold{C}$ we have $\langle \xi\rangle\leq 2|\zeta|$ and for all $s\in [0,1]$ that $$\langle \xi+s\zeta\rangle \leq \jap\xi+ |\zeta|\leq 3 |\zeta|.$$
Hence the definition of $R_s$ in \eqref{defn: def av r} and \eqref{useful derivatives}, \eqref{derivative of japs}, yield that
$$|(-\partial^{\gamma}_z) R_s|\lesssim |\zeta|^{\max(m',0)+\delta|\beta|+2M\delta +\delta|\gamma|}  \jap\xi^{m}(1+\langle\xi\rangle^{2\delta} |z|^2)^{-M}.$$

Therefore integrating by parts and choosing $N$ so large that
$$\max(m',0)+\delta|\beta| +2M\delta-2N(1-\delta) <-n$$
and
$$m+\delta|\beta|+\max(m',0)+ 2M\delta-2N(1-\delta)+(1-\delta)n\leq m+m'-\rho|\alpha|+\delta|\beta|$$ (observe once again that $\delta<1$), we obtain
\begin{equation*}
\begin{split}
\Big|\int_{\bold{C}} \int_{\R^n}e^{-i z\cdot\zeta}\,  R_s \dd z\ddd \zeta \Big| & = \Big|\int_{\bold{C}} \int_{\R^n}e^{-i z\cdot\zeta}\, |\zeta|^{-2N}\,  (-\Delta_z)^N R_s \dd z\ddd \zeta \Big|\\
& \lesssim \jap\xi ^{m-n\delta }\int_{\bold{C}} |\zeta|^{\max(m',0)+\delta|\beta|+2M\delta-2N(1-\delta)}  \ddd \zeta \\
& \lesssim \jap\xi ^{m+m'-\rho|\alpha|+\delta|\beta|}.
\end{split}
\end{equation*}
This concludes the proof.
\end{proof}

\section{Decomposition of the FIOs}\label{SSS decomposition}

In connection to the study of the $L^p$-regularity of FIOs, based on an idea of C. Fefferman {\cite{Feff}}, Seeger, Sogge and Stein \cite{SSS} introduced a second dyadic decomposition superimposed on a preliminary Littlewood-Paley decomposition, in which each dyadic shell $\left \{2^{j-1}\leq
\vert \xi\vert\leq 2^{j+1}\right \}$ (as in Definition \ref{def:LP}) is further partitioned into truncated cones of thickness roughly $2^{j/2}$ and one can prove that $O\big(2^{j(n-1)/2}\big)$
such elements are needed to cover one shell. Since we are dealing with FIOs with amplitudes in general H\"ormander classes, the constructions in \cite{SSS} have to be generalised to this setting. For instance we need to adapt the influence set associated to an FIO to the more general amplitudes at hand, and get desirable estimates of the size of the aforementioned influence set as well as a lower bound involving the phase function of the FIO.
Our presentation here is essentially self-contained.
\begin{defn}\label{def:LP2}
For each $j\in\N$ we fix a collection of unit
vectors $\big \{\xi^{\nu}_{j}\big \} $ that satisfy the following two conditions.
\begin{enumerate}
\item [$(i)$]$ \big | \xi^{\nu}_{j}-\xi^{\nu'}_{j} \big |\geq 2^{-j/2},$ if $\nu\neq \nu '$.
\item [$(ii)$]If $\xi\in\mathbb{S}^{n-1}$, then there exists a $
\xi^{\nu}_{j}$ so that $\big \vert \xi -\xi^{\nu}_{j} \big \vert
<2^{-j/2}$.
\end{enumerate}
One can take a collection $\{\xi_{j}^{\nu}\}$ which is maximal with respect to the first property and there are at most $O\big(2^{j (n-1)/2}\big)$ elements in the collection $\{\xi_{j}^{\nu}\}$.\\

Let $\Gamma^{\nu}_{j}$ denote the cone in the $\xi$-space whose
 central direction is $\xi^{\nu}_{j}$, i.e.
\begin{equation}\label{eq:gammajnu}
\Gamma^{\nu}_{j}
:=\set{ \xi\in\R^n :\,
\Big| \frac{\xi}{\vert\xi\vert}-\xi^{\nu}_{j}\Big|\leq 2\cdot 2^{-j/2}}.
\end{equation}
One also defines
$$\chi_j^\nu := \frac{\eta_j^\nu}{ \sum_{\nu} \eta_j^\nu},$$
where
\m{\eta_j^\nu (\xi) := \phi\brkt{2^{j/2}\brkt{\frac {\xi}{|\xi|}-\xi_j^\nu}}}
and $\phi$ is a non-negative function in $ \mathcal C_c^\infty(\R^n)$ with $\phi(u)=1$ for $|u|\leq 1$ and $\phi(u)=0$ for $u\geq 2$.
\end{defn}
We have the following lemma:
\begin{Lem}\label{lem:chijnu}
The functions $\chi_j^\nu \in \mathcal C^\infty(\R^n\setminus \{0\} )$ and are supported in the cones $\Gamma_j^\nu$. They sum to $1$ in $\nu$\emph:
\begin{equation*}
\sum _{\nu}\chi^{\nu}_{j}(\xi) =1,\quad \text{ for all } j\text { and } \xi\neq 0
\end{equation*}
and moreover they satisfy the estimates
\nm{eq:quadrseconddyadcond}{\abs{\d_\xi^\alpha \chi_j^\nu (\xi) }\lesssim
2^{j|\alpha|/2}
\abs\xi^{-\abs\alpha},}
for all multi-indices $\alpha$ and
\nm{eq:quadrseconddyadcond2}{ \left \vert \partial
^{N}_{\xi_{1}}\chi^{\nu}_{j}(\xi) \right \vert\leq C_{N}
\vert \xi\vert ^{-N}, \quad  \text{for} \quad N\geq 1,}
if one chooses the axis in $\xi$-space such that $\xi_1$ is in the direction of $\xi^{\nu}_{j}$ and \linebreak $\xi':=(\xi _2 , \dots , \xi_{n})$ is perpendicular to $\xi^{\nu}_{j}$.
\end{Lem}

\begin{proof}
In proving \eqref{eq:quadrseconddyadcond} we note that the argument of $\eta_j^\nu$ contains a factor of $2^{j/2}$ followed by something that is homogeneous of degree zero. Hence $\alpha$ derivatives yield a factor of $2^{j|\alpha|/2}$ and a function that is homogeneous of degree $-\abs\alpha$. To prove \eqref{eq:quadrseconddyadcond2} one observes that in the support of $\chi^\nu_j$ one can write $$\partial_{\xi_1}= \partial_{r}+O(2^{-j/2}) \cdot\nabla_\xi,$$
where $\partial_{r}$ is the radial derivative and $\partial_{r}^N \chi^\nu_j=0$  since $\chi^\nu_j$ is homogeneous of degree zero.
\end{proof}

If $\psi_j$ is chosen as in Definition \ref{def:LP} and we sum in both $j$ and $\nu$ one has
\begin{equation}\label{partofunity}
\psi_{0}(\xi)+\sum_{j=1}^{\infty}\sum_{\nu}\chi_{j}^{\nu}(\xi)\,\psi_{j}(\xi)=1, \quad \text{for all } \xi\in \mathbb{R}^{n}.
\end{equation}

When using the second dyadic decomposition we will split the phase $\varphi(x,\xi)-y\cdot\xi$ into two different pieces. The following lemma estimates one of the pieces.

\begin{Lem}\label{lem:h}
Define
$$h_j^\nu(x,\xi):= \varphi(x,\xi) - \xi\cdot\nabla_\xi \varphi(x,\xi_j^\nu).$$
Then on the support of $\chi_{j}^{\nu}(\xi)\psi_{j}(\xi)$, one has that
$$\abs{\partial_{\xi'}^\alpha h_j^\nu(x,\xi)} \lesssim \begin{cases}
\abs {\xi'} \abs \xi^{-1} & \text{ if } \abs\alpha=1 \\
\abs \xi ^{1-\abs\alpha} & \text{ if } \abs\alpha \geq 2
\end{cases} \,\leq 2^{-j\abs\alpha/2}$$
and
$$\abs{\partial_{\xi_1}^N h_j^\nu(x,\xi)}\lesssim \abs{\xi'}^2 \abs \xi ^{-N-1} \approx 2^{j} 2^{-j(N+1)}=2^{-jN}.$$
\end{Lem}
\begin{proof}
The proof is based on simple Taylor expansions and homogeneity considerations, see \cite[p. 407]{Stein}.
\end{proof}



In \cite{SSS} the authors define an \lq\lq influence set" associated to the SND phase function $\varphi$. We have to make a similar definition but it has to be fitted to the more general classes of amplitudes that we are considering here. To this end we have

\begin{defn}\label{def:influenceset}
Let $\bar y\in\R^n$ be the centre of a ball $B$ with radius $r$, and set $\mu:=\min(\rho, 1/2)$ with $\rho\in (0,1]$. Define
\m{\tilde R_j^\nu := \set{y\in \Rn : |y-\bar y|\leq c 2^{-\mu j}, \quad \abs{\pi_j^\nu (y-\bar y)} \leq c 2^{-\rho j}},}
where $\pi_j^\nu$ is the orthogonal projection in the direction $\xi_j^\nu$ and $c$ is a large constant depending on the size of the various Hessians of $\varphi$ but independent of $j$, to be specified later. We also define $R_j^\nu$ as the preimage of $\tilde R_j^\nu$ under the mapping $x\to \nabla_\xi \varphi(x,\xi_j^\nu)$, i.e.
\nm{eq:rectangles}{
R_j^\nu := \set{x\in \Rn : |\nabla_\xi \varphi(x,\xi_j^\nu)-\bar y|\leq c 2^{-\mu j}, \quad \abs{\pi_j^\nu (\nabla_\xi \varphi(x,\xi_j^\nu)-\bar y)} \leq c 2^{-\rho j}}.
}
Now set
\nm{eq:Bstar}{B^* = \bigcup_{2^{-j}\leq r }\bigcup_\nu R_j^\nu}
and recall that the number of $\nu$'s in the union above is $O\big(2^{j(n-1)/2}\big).$
\end{defn}
In this connection we have the the following estimates.

\begin{Lem}\label{measure of Bstar}
Let $\bar y\in\R^n$ be the centre of a ball $B$ with radius $r$ and let $B^*$ be defined as in \eqref{eq:Bstar}. Then
\begin{enumerate}
    \item[$(i)$] $B^*$ satisfies the estimate
        \m{\abs{B^*} \lesssim r^{\rho +(\mu -1/2) (n-1)}.}
    \item[$(ii)$] If $x\in \R^n \setminus B^*$, $y\in B(\bar y, r)$ and k is chosen such that $r\approx 2^{-k}$, then for $c$ in \eqref{eq:rectangles} large enough, we have
    \begin{equation}\label{eq:keypoint}
        2^{j\rho}\abs{(\nabla_\xi \varphi(x,\xi_j^\nu)-y)_1}+ 2^{j\mu}|(\nabla_\xi \varphi(x,\xi_j^\nu)-y)'|  \gtrsim 2^{(j-k)\rho /2}, \quad j \geq k.
    \end{equation}\hspace*{1cm}
\end{enumerate}

\end{Lem}

\begin{proof} \quad\\
$(i)$ Since $R_j^\nu$ is of size $O(2^{-\rho j})$ in the $\xi_j^\nu$-direction and $O(2^{-\mu j})$ in the other $n-1$ directions, we have
\m{\abs{B^*} \leq \sum_{2^{-j}\leq r} \sum_{\nu} \abs{R_j^\nu} \lesssim \sum_{2^{-j}\leq r} 2^{j(n-1)/2}\,2^{-j(\rho +\mu (n-1))}\lesssim r^{\rho +(\mu -1/2) (n-1)}.}

$(ii)$ Observe that \eqref{eq:keypoint} is equivalent to
\begin{equation*}
2^{(j+k)\rho/2} \abs{(\nabla_\xi \varphi(x,\xi_j^\nu)-  y)_1}
+2^{j \mu + (k-j)\rho/2}|(\nabla_\xi \varphi(x,\xi_j^\nu)-  y)'|
\gtrsim 1, \quad j \geq k.
\end{equation*}
Moreover, using that $j \geq k$, it is enough to show that
\begin{equation*}\label{lower bound estimate3}
2^{k\rho} \abs{(\nabla_\xi \varphi(x,\xi_j^\nu)-  y)_1}
+2^{k \mu}|(\nabla_\xi \varphi(x,\xi_j^\nu)-  y)'|
\gtrsim 1.
\end{equation*}
By construction, there exists a unit vector $\xi^{\nu'}_k$ with
$|\xi^{\nu'}_k -\xi^{\nu}_j|\leq 2^{-k/2}.$
Since
$\R^n \setminus B^*\subset \R^n \setminus R_k^{\nu'}$ then the assumption in ($ii$) yields that $x\in \R^n \setminus R_k^{\nu'}$ and therefore
\begin{equation*}
 2^{k\rho}|\pi_k^{\nu'} (\nabla_\xi \varphi(x,\xi_k^{\nu'})-\bar y)|+2^{k\mu}|\nabla_\xi \varphi(x,\xi_k^{\nu'})-\bar y| \geq c,
\end{equation*}
which implies that
\begin{equation*}
2^{k\rho}|(\nabla_\xi \varphi(x,\xi_k^{\nu'})-\bar y)_1|
+ 2^{k\mu}\abs{ (\nabla_\xi \varphi(x,\xi_k^{\nu'})-\bar y)'}
 \geq c/2,
\end{equation*}
for $c$ sufficiently large. On the other hand, for $|\xi^{\nu}_j-\xi^{\nu'}_k|\leq 2^{-k/2}$ one has
\begin{equation}\label{eq:SSStrick}
|(\nabla_\xi \varphi (x,\xi^{\nu}_j) - \nabla_\xi \varphi (x,\xi^{\nu'}_k))\cdot \xi^{\nu'}_k|\lesssim 2^{-k}.
\end{equation}
Indeed using the homogeneity we have
\begin{equation*}
\begin{split}
&(\nabla_\xi \varphi (x,\xi^{\nu}_j) - \nabla_\xi \varphi (x,\xi^{\nu'}_k))\cdot \xi^{\nu'}_k\\&\qquad = (\nabla_\xi \varphi (x,\xi^{\nu}_j) - \nabla_\xi \varphi (x,\xi^{\nu'}_k))\cdot (\xi^{\nu'}_k- \xi^{\nu}_j)+ \varphi (x,\xi^{\nu}_j) - \nabla_\xi \varphi (x,\xi^{\nu'}_k)\cdot \xi^{\nu}_j\\&\qquad = (\nabla_\xi \varphi (x,\xi^{\nu}_j) - \nabla_\xi \varphi (x,\xi^{\nu'}_k))\cdot (\xi^{\nu'}_k- \xi^{\nu}_j)+ h^{\nu'}_k(x,\xi^{\nu}_{j}),
\end{split}
\end{equation*}
where
$$h^{\nu'}_k(x,\xi) :=\varphi (x,\xi) - \nabla_\xi \varphi (x,\xi^{\nu'}_k)\cdot \xi.$$
Now the important fact is that $h^{\nu'}_k(x,\xi^{\nu'}_k)=0$ and therefore using the mean value theorem and the estimate $|\nabla_{\xi} h^{\nu'}_k(x,\xi)|\lesssim 2^{-k/2}$, one readily sees that $h^{\nu'}_k(x,\xi^\nu_j)= O(2^{-k}).$
For the term $ (\nabla_\xi \varphi (x,\xi^{\nu}_j) - \nabla_\xi \varphi (x,\xi^{\nu'}_k))\cdot (\xi^{\nu'}_k- \xi^{\nu}_j)$ we just use the mean value theorem, which concludes the proof of \eqref{eq:SSStrick}.\\

Now to show \eqref{eq:keypoint}, we use the triangle inequality and \eqref{eq:SSStrick} to obtain
\begin{align*}
& 2^{k\rho}\abs{(\nabla_\xi \varphi(x,\xi_j^\nu)-  y)_1}
+ 2^{k \mu}|(\nabla_\xi \varphi(x,\xi_j^\nu)-  y)'| \\
& \qquad =
2^{k\rho}\abs{\big(\nabla_\xi \varphi(x,\xi_j^\nu)-\bar y - (y-\bar{y})\big)_1}
+ 2^{k\mu}\big|\big(\nabla_\xi \varphi(x,\xi_j^\nu)-\bar y - (y-\bar{y})\big)'\big| \\
& \qquad \geq
2^{k\rho}\abs{\big(\nabla_\xi \varphi(x,\xi_j^\nu)-\bar y \big)_1} - 2^{k\rho}\abs{y_1-\bar{y}_1}
+ 2^{k\mu}\big|\big(\nabla_\xi \varphi(x,\xi_j^\nu)-\bar y \big)'\big|
\\
& \qquad\qquad- 2^{k\mu}|y'-\bar{y}'|
\\
& \qquad \geq
2^{k\rho}\big|\big(\nabla_\xi \varphi(x,\xi_j^\nu)-\bar y \big)_1\big|
+ 2^{k\mu}\big|\big(\nabla_\xi \varphi(x,\xi_j^\nu)-\bar y \big)'\big|
-2^{k\rho}2^{1-k} -2^{k\mu}2^{1-k}
\\
& \qquad =
2^{k\rho}\big|\big(\nabla_\xi \varphi(x,\xi_j^\nu) - \nabla_\xi \varphi(x,\xi_k^{\nu'}) + \nabla_\xi \varphi(x,\xi_k^{\nu'})-\bar y \big)_1\big|
\\
& \qquad \qquad + 2^{k\mu}\big|\big(\nabla_\xi \varphi(x,\xi_j^\nu) - \nabla_\xi \varphi(x,\xi_k^{\nu'}) + \nabla_\xi \varphi(x,\xi_k^{\nu'})-\bar y \big)'\big|
-2^{k\rho}2^{1-k}
\\
& \qquad\qquad-2^{k\mu}2^{1-k}
\\
& \qquad \geq
2^{k\rho}\big|\big( \nabla_\xi \varphi(x,\xi_k^{\nu'})-\bar y \big)_1\big|
 -
2^{k\rho}
\abs{\big(\nabla_\xi \varphi(x,\xi_j^\nu) - \nabla_\xi \varphi(x,\xi_k^{\nu'}) \big)_1}
\\
& \qquad \qquad + 2^{k\mu}\big|\big( \nabla_\xi \varphi(x,\xi_k^{\nu'})-\bar y \big)'\big|
- 2^{k\mu}\big|\big(\nabla_\xi \varphi(x,\xi_j^\nu) - \nabla_\xi \varphi(x,\xi_k^{\nu'}) \big)'\big|
\\
& \qquad \qquad  -2^{k\rho}2^{1-k} -2^{k\mu}2^{1-k}
\\
& \qquad \geq
 2^{k\rho}\big|\big( \nabla_\xi \varphi(x,\xi_k^{\nu'})-\bar y \big)_1\big|
+  2^{k\mu}\big|\big( \nabla_\xi \varphi(x,\xi_k^{\nu'})-\bar y \big)'\big|
\\
& \qquad \qquad
- 2^{k\rho}A 2^{-k}
- 2^{k\mu}C|\xi^{\nu}_j-\xi^{\mu}_k|
-2^{k\rho}2^{1-k} -2^{k\mu}2^{1-k}
\\
& \qquad \geq
 \frac{c}{2}
- 2^{k\rho}A  2^{-k}
- C 2^{k\mu}2^{-k/2}
-2^{k\rho}2^{1-k} -2^{k\mu}2^{1-k}
\\
& \qquad = \frac{c}{2}
- A 2^{-k(1-\rho)}
- C 2^{-k(1/2-\mu)}
-2^{1-k(1-\rho)} -2^{1-k(1-\mu)}
\\
& \qquad \geq \frac{c}{2}
- A
- C
-4.
\end{align*}
Therefore picking $c$ large enough we obtain \eqref{eq:keypoint}.
\end{proof}

Finally, having the partition of unity \eqref{partofunity} we can decompose $$T^{\varphi}_{a}= \sum_{j=0}^\infty T_j, $$
where
\begin{equation}\label{Tj}
    T_jf(x) := \sum_{\nu} T_j^\nu f(x)
\end{equation}
and
\m{T_j^\nu f(x) := \int_{\Rn} e^{i\varphi(x,\xi)}\,\chi_j^\nu(\xi)\,\psi_j(\xi)\,a(x,\xi)\,\widehat f (\xi) \ddd \xi .}

\section{\texorpdfstring{$L^p$}--results}\label{Lp section}
In this section we prove our main $L^p$-boundedness results for FIOs with general H\"ormander-class amplitudes. This generalizes the results of Seeger-Sogge-Stein in two ways. First of all this is a global regularity result and opposed to the local one in \cite{SSS}. Second, we consider all possible values of $\rho$'s and $\delta$'s as opposed to just $\rho\in [1/2, 1]$ and $\delta=1-\rho.$

\begin{Rem}\label{low freq rem}
We note that in what follows we can confine ourselves to the case of amplitudes $a(x, \xi)$ that vanish in a neighbourhood of the $\xi=0$. Indeed the contribution of the low-frequency portion $($i.e. the portion with compact $\xi$-support$)$ has been shown to be $L^p$-bounded for $1\leq p\leq \infty$  in \cite[Theorem 1.18]{DS}.
\end{Rem}

\subsection{Exotic amplitudes}
We start by proving the $L^p$-boundedness of exotic FIOs with amplitudes in $a\in S^m_{\rho,\delta}(\Rn)$ with $\rho=0$, $0 \leq \delta< 1$.

\begin{Th}\label{thm:LpResultexotic}
Let $n\geq 1$, $0\leq \delta<1$, $a\in S^m_{0,\delta}(\Rn)$, and assume that $\varphi\in\Phi^2$ is \emph{SND}. Then for
$$m= -n\Big | \frac{1}{p}-\frac{1}{2}\Big|-\frac{n\delta}{2}$$
the \emph{FIO} $T_a^\varphi$ is $L^p$-bounded for $1<p<\infty.$
\end{Th}
\begin{proof}
Using the discussion in Section \ref{RS globalisation}, we shall from now on assume that $T_a^{\varphi}$ is of the form
$$Tf(x):= \int_{\Rn} a(x,\xi)\,e^{i(\theta(x,\xi)+x\cdot\xi)}\, \widehat{f}(\xi)\ddd\xi,$$ where $\theta\in \Phi^1$ and $\theta(x,\xi)+x\cdot\xi$ is SND.\\

Since the result has already been proven for the case when $p=2$ (see Proposition \ref{basicL2}) it only remains to show that $T$ and its adjoint map $\mathscr{H}^p(\Rn)$ to $L^p(\Rn)$ continuously, for some $p<1$, and thereafter interpolate these with the $L^2$-boundedness.\\

Due to the atomic decomposition in Definition \ref{def:Hpatom} of an element of $\mathscr{H}^p(\Rn)$, we would need to show that
\nm{eq:crucialintegral1}{\int_{\Rn} |T\at(x)|^p\dd x}
is uniformly bounded for every $\mathscr{H}^p$-atom $\at$, where the atom is supported in the ball $B:=B(x_0 ,r)$. To prove the assertion in the case $2<p<\infty$ we also need the uniform boundedness of \eqref{eq:crucialintegral1} for the adjoint operator $T^*$. However since the proof is almost identical to the case of $T$, we confine ourselves to this case.\\

We split $\Rn$ into $2B$ and $\Rn\setminus 2B$ and start with the case of $2B$. By H\"older's inequality and the $L^2$-boundedness of $T$, we have
\nm{eq:combining1}{
\Vert T \at \Vert_{L^{p}(2B)}
&\lesssim \Vert T \at \Vert_{L^{2}(2B)} \|1\|_{L^{2p/(2-p)}(2B)} \lesssim \|\at\|_{L^{2}( \Rn)} r^{n(2-p)/2p}  \\
&\lesssim r^{n(p-2)/2p}\, r^{n(2-p)/2p} = 1.
}

We proceed to the boundedness of
$\Vert T \at \Vert_{L^{p}( \Rn\setminus 2B)}$. We consider a generic Littlewood-Paley piece of the kernel of $T$, which we denote by $S_j$ and note that the integral kernel of $S_j$ is given by
\begin{equation}\label{eq:2Bornot2B}
K_j(x,y)
:=\int_{\Rn} a_j(x,\xi)\,e^{i\theta(x,\xi)+i(x-y)\cdot\xi}\ddd\xi,
\end{equation}
where $a_j(x,\xi) := a(x,\xi)\,\psi_j(\xi)$, and $j\geq 1$ due to Remark \eqref{low freq rem}. We claim that
\nm{eq:kernelestimate1}{
\Vert (x-y)^\alpha \,\partial^{\beta} _{y}K_{j}(x,y)\Vert_{L^2_x(\Rn)}
&\lesssim 2^{j(  \vert \beta\vert+ m+n/2+n \delta/2)}.
}
Since differentiating \eqref{eq:2Bornot2B} $\beta$ times in $y$ will only introduce factors of the size $2^{j|\beta|},$ it is enough to establish \eqref{eq:kernelestimate1} for $\beta=0$. Now the global $L^2$-boundedness \eqref{eq:kernelestimate1} of the kernel can be formulated as the $L^2$-boundedness of a kernel of the form
\begin{equation*}
\tilde K^\alpha_j(x,x-y):= \int_{\R^n} a_j(x,\xi)\,(x-y)^\alpha\, e^{i\theta(x,\xi)+i(x-y)\cdot\xi}  \ddd\xi.
\end{equation*}
To this end, take ${\Psi}_j$ as in Definition \ref{def:LP}, integrate by parts and rewrite
\m{
&\tilde K^\alpha_j(x,x-y)
=\int_{\Rn} a_j(x,\xi)\,e^{i\theta(x,\xi)}\,
(-i)^{|\alpha|}\,\partial_\xi^\alpha e^{i(x-y)\cdot\xi} \ddd\xi\\
& \quad = i^{|\alpha|} \int_{\Rn} \partial_\xi^\alpha \Big[a_j(x,\xi)e^{i\theta(x,\xi)}\Big]\, e^{i(x-y)\cdot\xi} \,{\Psi}_j (\xi)\ddd\xi \\
& \quad = \sum_{\alpha_1 + \alpha_2=\alpha} \!\!C_{\alpha_1, \alpha_2}\int_{\Rn}
\partial_\xi^{\alpha_1} a_j(x,\xi) \,
\partial_\xi^{\alpha_2} e^{i\theta(x,\xi)} \,e^{i(x-y)\cdot\xi}\, {\Psi}_j (\xi)\ddd\xi \\
& \quad = \!\!\sum_{\substack{\alpha_1 + \alpha_2=\alpha \\ \lambda_1 + \dots + \lambda_r = \alpha_2} }\!\!\!\!
C_{\alpha_1, \alpha_2, \lambda_1, \dots \lambda_r} \int_{\Rn}
\partial_\xi^{\alpha_1} a_j(x,\xi) \\
& \quad \qquad \qquad   \times
\partial_\xi^{\lambda_1}\theta(x,\xi)
\cdots
\partial_\xi^{\lambda_r}\theta(x,\xi)
\,e^{i\theta(x,\xi)} \,e^{i(x-y)\cdot\xi}\,{\Psi}_j (\xi) \ddd\xi \\
& \quad =\!\! \sum_{\substack{\alpha_1 + \alpha_2=\alpha \\ \lambda_1 + \dots + \lambda_r = \alpha_2} }\!\!\!\!
C_{\alpha_1, \alpha_2, \lambda_1, \dots \lambda_r}
2^{j(m+{n\delta}/2)}  \!\!
\int_{\Rn} b_j^{\alpha_1,\alpha_2, \lambda_1, \dots, \lambda_r}(x,\xi)\,
 e^{i\theta(x,\xi)+ix\cdot\xi}\, e^{-iy\cdot\xi}\,{\Psi}_j (\xi) \ddd\xi \\
& \quad =:\!\! \sum_{\substack{\alpha_1 + \alpha_2=\alpha \\ \lambda_1 + \dots + \lambda_r = \alpha_2} }
\!\!\!\!C_{\alpha_1, \alpha_2, \lambda_1, \dots \lambda_r}
2^{j(m+{n\delta}/2)} \,
S_{j}^{\alpha_1, \alpha_2, \lambda_1, \dots \lambda_r}(\tau_{-y}\Psi_j^\vee)(x),
}
where $S_{j}^{\alpha_1, \alpha_2, \lambda_1, \dots \lambda_r}$ is an FIO with  the phase function $\theta(x,\xi)+x\cdot\xi$ and amplitude $b_j^{\alpha_1,\alpha_2, \lambda_1, \dots, \lambda_r}(x,\xi)$ given by
\begin{equation*}
b_j^{\alpha_1,\alpha_2, \lambda_1, \dots, \lambda_r}(x,\xi)
 := 2^{-j(m+{n\delta}/2)}\, \partial_\xi^{\alpha_1} a_j(x,\xi) \,
\partial_\xi^{\lambda_1}\theta(x,\xi)
\dots
\partial_\xi^{\lambda_r}\theta(x,\xi).
\end{equation*}
Moreover  $|\lambda_j| \geq 1$ and $\tau_{-y}$ is a translation by $-y$.\\

We observe that
$b_j^{\alpha_1,\alpha_2, \lambda_1, \dots, \lambda_r}(x,\xi) \in S^{-n\delta/2}_{0,\delta}(\Rn)$ uniformly in $j$, since $a \in S^{m}_{0,\delta}(\Rn)$ and $\theta\in \Phi^1$.\\

Therefore by Proposition \ref{basicL2}, $S_{j}^{\alpha_1, \alpha_2, \lambda_1, \dots \lambda_r}$ is an $L^2$-bounded FIO, so
\begin{align*}
\Vert \tilde K^\alpha_j(x,x-y)\Vert_{L^2_x(\Rn)}
& \lesssim \!\!\sum_{\substack{\alpha_1 + \alpha_2=\alpha \\ \lambda_1 + \dots + \lambda_r = \alpha_2} }\!\! 2^{j(m+{n\delta}/2)}
\|S_{j}^{\alpha_1, \alpha_2, \lambda_1, \dots \lambda_r}(\tau_{y}\Psi_j^\vee)\|_{L^2(\Rn)} \\
& \lesssim 2^{j(m+{n\delta}/2)} \| \Psi_j\|_{L^2(\Rn)}
\lesssim 2^{j(   m+n/2+{n\delta}/2)},
\end{align*}
which proves \eqref{eq:kernelestimate1}.\\

Now, the estimate in \eqref{eq:kernelestimate1} yields that for any integer $M$, if one sums over $\abs\alpha\leq M$,
\nm{eq:newone}{
\norm{ (1+ \vert x-y\vert)^M\, K_{j}(x,y)}_{L^2_x( \Rn)}&\lesssim 2^{j(n/2+ m+{n\delta}/2)}.
}

We now observe that for $t\in [0,1],$ $x\in  \Rn\setminus 2B$ and $y\in B$, one has
\begin{equation}\label{eq:estimateybar}
    \vert x-\bar{y}\vert \lesssim \vert x-\bar{y}-t(y-\bar{y})\vert.
\end{equation}

Next we introduce
$$g(x):=\Big(1+ |x-\bar{y}|\Big)^{-M},$$
where $M > n/q$ and $1/q=1/p-1/2$. The H\"older and the Minkowski inequalities together with \eqref{eq:newone} and \eqref{eq:estimateybar} (with $t=1$) yield
\begin{align}\label{eq:keyestimate1}
 \|S_j\at \|_{L^p( \Rn\setminus 2B)}
& = \Big\| \int_{B} K_{j}(x,y)\,\at(y)\dd y \Big\| _{L^p_x(\Rn\setminus 2B)} \\
& \leq \Big\|  \frac{1}{g(x)}
\int_{B}K_{j}(x,y)\,\at(y)\dd y \Big\| _{L^2_x( \Rn\setminus 2B)}
\,\Vert g \Vert _{L^q( \Rn)}  \nonumber \\
&\lesssim
  \int_{B} \Big\|\frac{1}{g(x)}\,K_{j}(x,y)\,\at(y)\Big\|_{L^2_x( \Rn\setminus 2B)}  \dd y \nonumber \\
 &\lesssim   \int_{B} \vert \at(y) \vert \, \Big\| (1+ |x-{y}|)^{M} \,   K_{j}(x,y)\Big\| _{L^2_x( \Rn\setminus 2B)} \!\dd y \nonumber \\
&\lesssim r^{n-n/p}\,   2^{j(n/2+ m+{n\delta}/2)}\lesssim r^{n-n/p}\,2^{j(n-n/p)}, \nonumber
\end{align}
since $ m=-n(1/p-1/2)-n\delta/2$.\\

On the other hand, taking $N:=[n(1/p-1)]$ (note that $N > n/p-n-1$),
a Taylor expansion of the kernel at the point $y=\overline{y}$ yields that
 \m{
K_{j}(x,y)
& = \sum_{ |\beta| \leq N} \frac{(y-\bar y)^\beta}{\beta!}\,  \partial^\beta_y (K_{j}(x, y))_{|_{y=\bar y}} \\
& \qquad + (N+1) \sum_{ |\beta| = N+1}  \frac{(y-\bar y)^\beta}{\beta!}  \int_0^1 (1-t)^N \,\partial^\beta_y (K_{j}(x,y))_{|_{y=\bar y+t(y-\bar y)}} \dd t
}
and due to vanishing moments of the atom in Definition \ref{def:Hpatom}, $iii)$, we may express the operator as
\begin{align*}
S_j\at(x)   =(N+1)\!\!\! \sum_{ |\beta| = N+1} \int_{B}  \int_0 ^1 \frac{(y-\bar{y})^\beta}{\beta!}\, (1-t)^N\,
\partial^\beta_y (K_{j}(x,y))_{|_{y=\bar y+t(y-\bar y)}}
\at(y) \dd t\dd y.
\end{align*}

Noting that $\abs{ (y-\bar{y})^\beta }\lesssim r^{N+1}$ and applying the same procedure as above together with estimates \eqref{eq:kernelestimate1} and \eqref{eq:estimateybar}, we obtain
\begin{equation}\label{eq:keyestimate2}
\Vert S_j \at \Vert _{L^p(\Rn\setminus 2B)}
\lesssim r^{N+1-n/p+n}\,2^{j(N+1+ m+n/2+{n\delta}/2)}\lesssim r^{{N} +1+n-n/p }\,2^{j({N}+1+n-n/p)}.
\end{equation}

Now we split the proof in two different cases, namely when the radius $r$ of the support of the atom $\at$ is less than or greater or equal to one.\\

For $r\geq 1$, \eqref{eq:keyestimate1} yields that
\begin{equation*}\label{eq:largeballs}
\| T \at \|_{L^p(\Rn\setminus 2B)}^p
\lesssim \sum_{{j=1}}^\infty \| S_j \at \|_{L^p(\Rn\setminus 2B)}^p \lesssim \sum_{{j=1}}^\infty r^{np-n}\,2^{j(np-n)} \lesssim 1.
\end{equation*}

Assume now that $r < 1$.
Choose $\ell \in \Z_+$ such that
$2^{-\ell-1} \leq r < 2^{-\ell }$. Using the facts that $2^{-\ell}\approx r$, $N+1+n-n/p>0$, $n-n/p<0$, together with \eqref{eq:keyestimate1} and \eqref{eq:keyestimate2} we conclude that
\begin{align*}
\| T \at\|_{L^p(\Rn\setminus 2B)}^p
& \lesssim \sum_{{j=1}}^{\ell} \brkt{r^{N+1+n-n/p} \, 2^{j(N+1+n-n/p)}}^p
	+ \sum_{j=\ell +1}^{\infty} \brkt{r^{n-n/p} \, 2^{j(n-n/p)}}^p \\
& \lesssim  \brkt{r^{N+1+n-n/p} \, 2^{\ell (N+1+n-n/p)}}^p
	+ \brkt{r^{n-n/p} \, 2^{\ell(n-n/p)}}^p \\
& \approx  \brkt{r^{N+1+n-n/p} \, r^{-(N+1+n-n/p)}}^p
	+ \brkt{r^{n-n/p} \, r^{-(n-n/p)}}^p\\
& \approx  1.
\end{align*}
Putting this together with \eqref{eq:combining1}, yields the uniform boundedness of \eqref{eq:crucialintegral1}.\\

The proof of the adjoint case is identical, except for the fact that  \eqref{eq:kernelestimate1} becomes
\m{
\tilde K^\alpha_j(y,x-y):= \int_{\R^n} a_j(y,\xi)\,(x-y)^\alpha\, e^{-i\theta(y,\xi)+i(x-y)\cdot\xi}  \ddd\xi
}
and when applying $\beta$ derivatives in the $y$-variable the $y$-dependence in both arguments has to be taken into consideration.
\end{proof}

It is also evident that Theorem \ref{thm:LpResultexotic} yields the $L^p$-boundedness of pseudodifferential operators with exotic symbols and thereby completes the investigation in \cite{AH}.

\subsection{Classical amplitudes}
We proceed by proving a global $L^p$-boundedness result for classical FIOs with amplitudes in $a\in S^m_{\rho,\delta}(\Rn)$ with $0 < \rho \leq 1$, $0 \leq \delta< 1$.\\

Before doing that, we need the following lemma which provides $\mathscr{H}^q-L^2$ estimates for FIOs with amplitudes in general H\"ormander classes.

\begin{Lem}\label{lemma:HLS-lemma}
Let $m_1\leq 0$, $n\geq 1,$ $\rho\in [0,1]$, $\delta\in [0,1)$ and
\begin{equation*}
m:=m_1-n \max\brkt{0,\frac{\delta-\rho}2}.
\end{equation*}
Suppose that $a\in S^m_{\rho,\delta}(\Rn)$ and that $a(x,\xi)$ vanishes in a neighborhood of $\xi=0$. Also, let $\varphi$ be an \emph{SND} phase function in the class $\Phi^2$.
Then $T_a^\varphi$, defined in \eqref{eq:FIO}, satisfies
\nm{eq:L2toLq}{
\|T_a^\varphi f\|_{L^2(\Rn)}
\lesssim
\|f\|_{\mathscr{H}^{{2n}/{(n-2 m_1)}}(\Rn)}.}
Also for the adjoint operator one has
\nm{eq:L2toLqadj}{
\|(T_a^{\varphi})^* f\|_{L^2(\Rn)} \lesssim
\|f\|_{\mathscr{H}^{{2n}/{(n-2 m_1)}}(\Rn)}.}
\end{Lem}

\begin{proof}
Since the operator $T_a^\varphi (1-\Delta)^{-m_1/2}$ is an FIO with the phase $\varphi$ and an amplitude in $S^{-n \max(0,(\delta-\rho)/2)}_{\rho,\delta}(\Rn)$ it is $L^2$-bounded by Proposition \ref{basicL2}. This $L^2$-boundedness together with the estimates for the Bessel potential operators reformulated in terms of embedding of Triebel-Lizorkin spaces (see \cite[Corollary 2.7]{Triebel4}) yield
\m{
\|T_a^\varphi f\|_{L^2(\Rn)}
& = \|T_a^\varphi (1-\Delta)^{-m_1 /2}(1-\Delta)^{m_1 /2}f\|_{L^2(\Rn)} \\
&\lesssim \|(1-\Delta)^{m_1 /2 }f\|_{L^2(\Rn)}
\lesssim \|f\|_{\mathscr{H}^{q}(\Rn)},}
with $1/q-1/2= -m_1/n,$ which proves \eqref{eq:L2toLq}. Here observe that the choice of the range of $m_1$ implies that $0< q\leq 2$.\\

Next we prove \eqref{eq:L2toLqadj}. By Theorem \ref{thm: main composition} the composition  $ (1-\Delta)^{-m_1 /2} T_a^\varphi$ is an FIO with the phase $\varphi$ and an amplitude in $S^{-n \max(0,(\delta-\rho)/2)}_{\rho,\delta}(\Rn)$, and therefore $L^2$-bounded. Finally, observing that $$\Vert (T_a^\varphi)^*(1-\Delta)^{-m_1 /2}\Vert_{L^2 \to L^2} = \Vert  ((1-\Delta)^{-m_1 /2}T_a^\varphi)\Vert_{L^2 \to L^2} ,$$ one can proceed as above.
\end{proof}

Now we are ready to state an prove our main $L^p$-estimate for FIOs with general classical H\"ormander-type amplitudes.

\begin{Th}\label{thm:LpResult}
Let {$n\geq1$}, $a\in S^m_{\rho,\delta}(\Rn)$, $\varphi$ be an \emph{SND} phase function in the class $\Phi^2$ and let $T_a^\varphi$ be given as in \emph{Definition \ref{def:FIO}}. For the case $0 < \rho \leq 1$, $0 \leq \delta< 1$, $\mu:=\min(\rho, 1/2)$ and
$$ m=
\set{\rho-n +\Big(\mu-\frac{1}{2}\Big) (n-1) }
\abs{\frac{1}{p}-\frac{1}{2}} - n\max\brkt{0,\frac{\delta-\rho}2},$$
the \emph{FIO} $T_a^\varphi$ is $L^p$-bounded for $1<p<\infty.$
\end{Th}

\begin{Rem}
We note that for $\rho\in [1/2, 1]$ and $\delta= 1-\rho$, the order $m$ in the theorem above is equal to $-(n-\rho)|1/p-1/2|$ which is sharp and is the same order of decay as in \cite{SSS} $($in the maximal rank case, i.e. the case when the rank of the Hessian $($in $\xi$$)$ of the phase is equal to $n-1$\emph{)}. Moreover if the amplitude $a$ is assumed to have compact spatial support then one can replace the \emph{SND} condition in the theorem above by the non-degeneracy condition of \emph{Definition \ref{nondegeneracy}}.
\end{Rem}

\begin{Rem}\label{Rem:rho0}
In \emph{Theorem \ref{thm:LpResult}} it is not possible to consider the case $\rho=0$ for several reasons. First, the definitions of the rectangles in \eqref{eq:rectangles} would not make sense. Second, the choice of $M$ in the proof of \eqref{eq:Goal2} below would not be possible. Third, the choice of $L> n/\rho$ in \eqref{Ruz-Sug ball estim} would be problematic.
\end{Rem}

\begin{Rem}
We observe that if $0\leq \rho\leq 1$ then $S^m_{\rho, \delta}(\Rn)\subset S^{m}_{0,\delta}(\Rn)$. This behoves us to compare the orders of decay i.e. the $m$'s  in \emph{Theorem \ref{thm:LpResultexotic}} and \emph{Theorem \ref{thm:LpResult}}, which we denote by $m_1$ and $m_2$ respectively. For the sake of discussion let us compare the $m$'s that are required for the $\mathscr{H}^1-L^1$ boundedness, and so assume that $p=1$. In the case of $\delta\geq\rho$ and $\rho<1/2$, then we have
$$m_1= - \frac{n}{2} - \frac{n\delta}{2},$$
and
$$m_2=\frac{\rho}{2} -\frac{n}{2} +\brkt{\rho-\frac{1}{2}}\,\frac{n-1}{2}- n\,\frac{\delta-\rho}{2}.$$ Here we see that $m_2 >m_1$ iff $(1-4\rho)n<1.$
Therefore if $1/4\leq \rho<1/2$ and $\delta\geq\rho$ then \emph{Theorem \ref{thm:LpResult}} is an improvement of \emph{Theorem \ref{thm:LpResultexotic}}. On the other hand, if $\delta < \rho< 1/2$ and $\rho+\delta \geq 1/2$ then $m_1$ is the same as above and
$$m_2=\frac{\rho}{2} -\frac{n}{2} +\brkt{\rho-\frac{1}{2}}\,\frac{n-1}{2},$$
and we see once again that $m_2>m_1.$ Therefore even in this case \emph{Theorem \ref{thm:LpResultexotic}} provides an improvement.
\end{Rem}

\begin{proof}[Proof of \emph{Theorem \ref{thm:LpResult}}]
For $n=1$, it is well known that FIOs are special cases of pseudodifferential operators and hence the result follows from the corresponding theory for those operators (see e.g. \cite{Stein}). Therefore, from now on we concentrate on the case $n\geq 2.$
 We will initially assume that $a(x,\xi)$ is supported in a fixed compact set in the $x$-variable. This will however be removed later on in the proof. Since the result has already been proven for the case when $p=2$ in Proposition \ref{basicL2}, the only thing that is left to prove is that $T_a^\varphi$ and its adjoint map $\mathscr{H}^1(\Rn)$ to $L^1(\Rn)$ continuously, when $a(x,\xi)\in S^m_{\rho, \delta}(\R^n)$ with
 $$m=\frac{\rho-n}{2} +\Big(\mu-\frac{1}{2}\Big)\frac{(n-1)}{2}
 - n\max\brkt{0,\frac{\delta-\rho}2}.$$
 Due to the atomic decomposition in Definition \ref{def:Hpatom} of a member of $\mathscr{H}^1(\Rn)$ we  need to show that
\nm{eq:crucialintegral2}{\int_{\Rn} |T_a^\varphi \at(x)|\dd x}
is uniformly bounded  for every $\mathscr{H}^1$-atom $\at$, where the atom is supported in the ball $B(x_0 ,r)$.\\

\textbf{Step 1 - Estimates of $\mathbf{\|T_a^\varphi \at\|_{L^1(\R^n)}}$ when $\mathbf{r\leq 1}$}

Recalling the set $B^*$ in \eqref{eq:Bstar}, we split \eqref{eq:crucialintegral2} into two pieces, namely
\m{
\|T_a^\varphi \at\|_{L^1(\Rn)}
= \|T_a^\varphi \at\|_{L^1(B^*)}
+\|T_a^\varphi \at\|_{L^1(\Rn \setminus B^*)}
=: \operatorname I+\operatorname {II}.}

Using the first part of Lemma \ref{measure of Bstar}, the Cauchy-Schwarz inequality, and Lemma \ref{lemma:HLS-lemma} we can deduce that $\operatorname I$ is uniformly bounded. Indeed take $\mathfrak b = |B|^{1-  1/q}\at$ with $q={2n}/{(n-2m_1)}$. Note that $\mathfrak b$ is then an $\mathscr{H}^{q}$-atom with $\Vert \mathfrak b\Vert_{\mathscr{H}^{q}(\Rn)}= 1$. Thus,
\begin{align*}
\operatorname I
& \lesssim
r^{\rho/2 + (\rho/2-1/4)(n-1)} \,
\|T_a^\varphi  \at\|_{L^2(\Rn)}  \lesssim
r^{\rho/2 + (\rho/2-1/4)(n-1)}\,
\| \at\|_{\mathscr{H}^{{2n}/{(n-2m_1)}}(\Rn)}\\
& = r^{\rho/2 + (\rho/2-1/4)(n-1)+n/q-n}\,
\| \mathfrak b\|_{\mathscr{H}^{{2n}/{(n-2m_1)}}(\Rn)}=
r^{\rho/2 + (\rho/2-1/4)(n-1)+n/q-n} =1,
\end{align*}
provided that
$$m_1 =
\frac{\rho-n}{2} + \brkt{\mu - \frac{1}{2}}\frac{n-1}{2}.$$

Now we have to deal with the last and most complicated part of the proof, that is the boundedness of $\operatorname {II}$. To do this, we use the decomposition \eqref{Tj}. Denoting the kernel of $T_j$ by $K_j$, $j \geq 1$ (recall Remark \ref{low freq rem}), we would first like to prove that
\begin{equation}\label{eq:Pregoal}
\int_{\R^n}\abs{\nabla_y K_j(x,y)}\dd x\lesssim 2^j,
\end{equation}
which immediately yields
\begin{equation}\label{eq:Goal1}
\int_{\Rn} |K_j(x,y) - K_j(x,z)|   \dd x
\lesssim 2^{j}\, |y-z|.
\end{equation}
To justify \eqref{eq:Pregoal}, take
\m{K_j^\nu(x,y) := \int_{\Rn} e^{i\varphi(x,\xi)-iy\cdot\xi }\,\chi_j^\nu(\xi)\,\psi_j(\xi)\,a(x,\xi)  \ddd \xi,}
set
$$h^{\nu}_j(x,\xi):= \varphi(x,\xi) - \xi\cdot\nabla_\xi \varphi(x,\xi_j^\nu)$$
and rewrite
\m{\nabla_y K_j^\nu(x,y) = \int_{\Rn} e^{i\xi\cdot \nabla_\xi \varphi(x,\xi_j^\nu)-iy\cdot\xi }\,b_j^\nu(x,\xi)  \ddd \xi,}
with
$$b_j^\nu (x,\xi) := -i\xi\, a(x,\xi)\,\chi_j^\nu(\xi)\,\psi_j(\xi)\,e^{ih^{\nu}_j(x,\xi)}.$$ Define the differential operator
\m{L:= \brkt{1-2^{2j \rho }\partial_{\xi_1}^2}\brkt{1-{2^{2j\mu}\Delta_{\xi'}}},}
where $\mu:=\min(\rho, 1/2)$. It is clear that for $N\geq 1$
\m{L^N e^{i\xi\cdot \nabla_\xi \varphi(x,\xi_j^\nu)-iy\cdot\xi } &= e^{i\xi\cdot \nabla_\xi \varphi(x,\xi_j^\nu)-iy\cdot\xi }\brkt{1+2^{2j\rho}\abs{( \nabla_\xi \varphi(x,\xi_j^\nu)-y)_1}^2}^N \\ &\qquad\times \brkt{1+2^{2j\mu}\abs{( \nabla_\xi \varphi(x,\xi_j^\nu)-y)'}^2}^N.}
We investigate how differentiation in different directions affects $b_j^\nu$. Using the fact that $a\in S_{\rho,\delta}^m(\Rn)$, Lemmas \ref{lem:chijnu} and \ref{lem:h}, we deduce that
\begin{align*}
& \abs{2^{2j \rho N_1} \partial_{\xi_1}^{2N_1} 2^{2j \mu N_2} \partial_{\xi_k}^{2N_2}\Big (a(x,\xi)\,\chi_j^\nu(\xi)\,\xi\,\psi_j(\xi)\, e^{ih^{\nu}_j(x,\xi)}\Big )} \\
& \leq \sum _{\substack{ \alpha_1+\dots +\alpha_4=2N_1 \\
\vert \beta_1 +\dots +\beta_4\vert=2N_2}}
\abs{2^{2j \rho N_1}\,2^{2j \mu N_2} \,\partial_{\xi_1}^{\alpha_1} \partial_{\xi_k}^{\beta_1}a(x,\xi)\,\partial_{\xi_1}^{\alpha_2}\partial_{\xi_k}^{\beta_2}\chi_j^\nu(\xi)\,\partial_{\xi_1}^{\alpha_3}\partial_{\xi_k}^{\beta_3}(\xi\psi_j(\xi))\,\partial_{\xi_1}^{\alpha_4}\partial_{\xi_k}^{\beta_4}e^{ih^{\nu}_j(x,\xi)}} \\
&   \lesssim \sum _{\substack{ \alpha_1+\dots +\alpha_4=2N_1 \\
\vert \beta_1 +\dots +\beta_4\vert=2N_2}}
2^{2j \rho N_1}\,2^{2j \mu N_2}\, 2^{j(m -\alpha_1\rho-\abs{\beta_1}\rho)}\,2^{-j\alpha_2-j\abs{\beta_2}/2}\,2^{j-j\alpha_3-j\abs{\beta_3}}\,2^{-j\alpha_4-j\abs{\beta_4}/2}\\
&  \lesssim \sum _{\substack{ \alpha_1+\dots +\alpha_4=2N_1 \\
\vert \beta_1 +\dots +\beta_4\vert=2N_2}}
2^{2j \rho N_1}\,2^{2j \mu N_2}\, 2^{j(m -\alpha_1\rho-\abs{\beta_1}\mu)}\,2^{-j\alpha_2\rho -j\abs{\beta_2}\mu}\,2^{j-j\alpha_3\rho -j\abs{\beta_3}\mu}\,2^{-j\alpha_4\rho -j\abs{\beta_4}\mu} \\
&  =
2^j\,2^{2j \rho N_1}\,2^{2j \mu N_2}\, 2^{j(m -2N_1\rho-{2N_2}\mu)} = 2^{j+jm}.
\end{align*}
This proves that
\nm{eq:modifiedsymbolestimate}{\abs{L^N b_j^\nu(x,\xi)}
\lesssim 2^{j(m+1)}.}

Now using integration by parts
\m{\nabla_y K_j^\nu(x,y)  = \frac {
\displaystyle \int_{\Gamma_j^\nu}e^{i\xi\cdot \nabla_\xi \varphi(x,\xi_j^\nu)-iy\cdot\xi }L^N b_j^\nu(x,\xi)\ddd \xi}{\brkt{1+2^{2j\rho }\abs{( \nabla_\xi \varphi(x,\xi_j^\nu)-y)_1}^2}^N\brkt{1+2^{2j\mu}\abs{( \nabla_\xi \varphi(x,\xi_j^\nu)-y)'}^2}^N},}
where $\Gamma^{\nu}_j$ defined as in \eqref{eq:gammajnu} is the support of $b_j^\nu(x,\xi)$.
Let $\widehat{g}_j^\nu$ be a function that is constantly equal to one on the $\xi$-support of $b_j^\nu$ and set $\mathbf t(x):= \nabla_\xi \varphi(x,\xi_j^\nu)$. Define
\m{
S_{j,y}^{\nu,N}  {g}_j^\nu (x) := 2^{-j-j m_1}\int_{\Gamma_j^\nu} e^{ix\cdot\xi  }\,\set{L^N b_j^\nu\brkt{\mathbf t^{-1}(x+y),\xi}}\,\widehat{g}_j^\nu(\xi)\ddd\xi.
}
Then because of \eqref{eq:modifiedsymbolestimate}, the choice of $m_1$, and that $\mathbf t$ is a diffeomorphism, $S_{j,y}^{\nu,N}$ is a $\Psi$DO of order
$- n\max(0,(\delta-\rho)/2)$ and hence {$L^2$-bounded}, by Proposition \ref{basicL2}, uniformly in $y$ and $j$. Observe that $\nabla_yK_j^\nu(x,y)$ can be rewritten as
\m{
\nabla_y K_j^\nu(x,y) & = \frac {2^{j+jm_1}\,(S_{j,y}^{\nu,N}{g}_j^\nu)(\nabla_\xi \varphi(x,\xi_j^\nu)-y)}{\brkt{1+2^{2j\rho }\abs{( \nabla_\xi \varphi(x,\xi_j^\nu)-y)_1}^2}^N\brkt{1+2^{2j\mu}\abs{( \nabla_\xi \varphi(x,\xi_j^\nu)-y)'}^2}^N}.
}

Now using the compact $x$-support, Cauchy-Schwarz inequality and that $\mathbf t(x)$ is a diffeomorphism, we have
\m{
&\brkt{ \int_{\R^n}\abs{\nabla_y K_j^\nu(x,y)}\dd x}^2   \\
&\qquad \lesssim \int_{\R^n}\frac{2^{2j+2jm_1}\,
\|S_{j,y}^{\nu,N}{g}_j^\nu\|_{L^2(\Rn)}^2}{\brkt{1+2^{2j\rho}\abs{( \nabla_\xi \varphi(x,\xi_j^\nu)-y)_1}^2}^{2N}\brkt{1+2^{2j\mu}\abs{( \nabla_\xi \varphi(x,\xi_j^\nu)-y)'}^2}^{2N}} \dd x   \\
& \qquad \lesssim  2^{2j+2jm_1-j(\rho+(n-1)\mu)}\,
\|{g}_j^\nu\|_{L^2(\Rn)}^2
\lesssim 2^{ 2j+2jm_1-j(\rho+(n-1)\mu) +j( (n-1)/2+1)}
\lesssim 2^{(3-n)j} .
}

Therefore, summing in $\nu$ and observing that since there are roughly $2^{j(n-1)/2}$ terms involved, we obtain
\m{\int_{\R^n}\abs{\nabla_y K_j(x,y)}\dd x\leq \sum_{\nu} 2^{j(3-n)/2}\lesssim
2^j,}
which is \eqref{eq:Pregoal}.\\

Our next goal is to show that
\begin{equation}\label{eq:Goal2}
\int_{\Rn \setminus B^*} |K_j(x,y)|   \dd x
\lesssim (2^{j} r)^{-1}, \qquad y \in B, \quad r \geq 2^{-j}.
\end{equation}

Indeed, a similar calculation as in the case of $\nabla_y K_j(x,y)$ and estimate \eqref{eq:keypoint} yield that, for any $M\geq 0$ one has
\begin{align*}
& \brkt{\int_{\Rn \setminus B^*}  \abs{K_j^\nu(x,y)}\dd x }^2   \\
& \qquad  \qquad \lesssim \int_{\Rn \setminus B^*} \frac{2^{2jm_1}\, \|S_{j,y}^{\nu,N}
{g}_j^\nu\|_{L^2(\Rn)}^2}{\brkt{1+2^{2j\rho}\abs{( \nabla \varphi(x,\xi_j^\nu)-y)_1}^2}^{2N+M/2-M/2}}\\
&\qquad \qquad \qquad \times \frac{1}{\brkt{1+2^{2j\mu}\abs{( \nabla \varphi(x,\xi_j^\nu)-y)'}^2}^{2N+M/2-M/2}} \dd x   \\
& \qquad \qquad \lesssim \int_{\Rn \setminus B^*} \frac{ 2^{2jm_1}\, \|S_{j,y}^{\nu,N}
{g}_j^\nu\|_{L^2(\Rn)}^2 }{\brkt{2^{j\rho}\abs{( \nabla \varphi(x,\xi_j^\nu)-y)_1}+2^{j\mu}\abs{( \nabla \varphi(x,\xi_j^\nu)-y)'}}^{M}} \\
& \qquad \qquad \qquad \times  \frac{1}{\brkt{1+2^{2j\rho}\abs{( \nabla \varphi(x,\xi_j^\nu)-y)_1}^2}^{2N-M/2}}\\
&\qquad\qquad \qquad \times \frac{1}{\brkt{1+2^{2j\mu}\abs{( \nabla \varphi(x,\xi_j^\nu)-y)'}^2}^{2N-M/2}} \dd x \\
& \qquad \qquad \lesssim \int_{\Rn \setminus B^*}   \frac{2^{-M(j-k)\rho/2}\,
2^{2jm_1}\,
\|S_{j,y}^{\nu,N}
{g}_j^\nu
\|_{L^2(\Rn)}^2}{\brkt{1+2^{2j\rho}\abs{( \nabla \varphi(x,\xi_j^\nu)-y)_1}^2}^{2N-M/2}}\\
&\qquad \qquad \qquad \times \frac{1}{\brkt{1+2^{2j\mu}\abs{( \nabla \varphi(x,\xi_j^\nu)-y)'}^2}^{2N-M/2}} \dd x \\
&\qquad \qquad \lesssim
2^{-M(j-k)\rho/2} \,
2^{-j(\rho+(n-1)\mu)}\,
2^{2jm_1}\,
\|
{g}_j^\nu
\|_{L^2(\Rn)}^2 \\
& \qquad \qquad \lesssim
2^{-M(j-k)\rho/2} \,
2^{ -j(\rho+(n-1)\mu)}\,
2^{2jm_1}\,
2^{j((n-1)/2+1) } \\
& \qquad \qquad = 2^{-M(j-k)\rho/2} \,2^{-j(n-1)}.
\end{align*}
Hence,
\begin{align*}
\int_{\Rn \setminus B^*}  \abs{K_j(x,y)}\dd x
& \leq \sum_{\nu} 2^{-j(n-1)/2}\,2^{-M(j-k)\rho/4}  \leq 2^{-M(j-k)\rho/4}
 \approx (2^j r)^{-M\rho/4}
\end{align*}
and taking $M:=4/\rho$ yields \eqref{eq:Goal2}.\\

Finally we write
\begin{equation*}
T_a^\varphi \at(x)= \sum_{2^j > r^{-1}} T_j \at (x)+ \sum_{2^j < r^{-1}} T_j \at(x),
\end{equation*}
where as before, $r$ is the radius of the support of the atom $\at$. Taking the $L^1$-norm, then property (\emph{iii}) of Definition \ref{def:Hpatom}, Minkowski's inequality, \eqref{eq:Goal1} and \eqref{eq:Goal2} yield that
\begin{align*}
\operatorname {II}
&=
\|T_a^\varphi \at\|_{L^1(\Rn\setminus B^*)}
\leq  \sum_{2^j <r^{-1}}
\|T_j \at\|_{L^1(\Rn \setminus B^*)}
+\sum_{2^j >r^{-1}}
\|T_j \at\|_{L^1(\Rn \setminus B^*)}\\
&\lesssim \sum_{2^j <r^{-1}} \int_B
\|K_j(x,y)-K_j(x,\bar y)\|_{L^1_x(\Rn)}\,\abs{\at(y)}\dd y \\
&\qquad +\sum_{2^j >r^{-1}} \int_B
\|K_j(x,y)\|_{L^1_x(\Rn)}\,\abs{\at(y)}\dd y \\
&\lesssim \sum_{2^j <r^{-1}} \int_B 2^jr\, r^{-n}\dd y +\sum_{2^j >r^{-1}} \int_B(2^j r)^{-1}\,r^{-n} \dd y \\
&= \sum_{2^j <r^{-1}}  2^jr+\sum_{2^j >r^{-1}}(2^j r)^{-1}
\lesssim 1.
\end{align*}
The corresponding proof of the  $\mathscr{H}^{1}-L^1$  boundedness of the adjoint $(T_a^\varphi)^*$ is similar to the one above with few modifications. First, \eqref{eq:L2toLq} has to be replaced by \eqref{eq:L2toLqadj}. Second, the $x$ and $y$ dependencies of the kernel are reversed. This means the following replacements:
\m{
\nabla_\xi\varphi(x,\xi_j^\nu) \longrightarrow x,\\
y \longrightarrow \nabla_\xi\varphi(y,\xi_j^\nu),\\
\bar y \longrightarrow \nabla_\xi\varphi(\bar y,\xi_j^\nu).
}
Otherwise the proof remains the same.
 \\

\textbf{Step 2 - Estimates of $\mathbf{\|T_a^\varphi \at\|_{L^1(\R^n)}}$ when $\mathbf{r\geq 1}$}\\
{Now we turn our attention to} atoms with supports in balls of radii $r\geq 1$. In this case, using the compact support of the amplitude $a$ and the $L^2$-boundedness of $T^{\varphi}_a$ (Proposition \ref{basicL2}) we have
\nm{eq:largeball}{
\|T_a^\varphi \at\|_{L^1(\Rn)}
\lesssim  \|T_a^\varphi \at\|_{L^2(\Rn)} \lesssim \|\at\|_{L^2(\Rn)}
\lesssim r^{-n/2}
\lesssim 1.}

To prove the boundedness of the adjoint $(T_a^\varphi)^*$, we split the $L^1$-norm into two pieces, namely
\begin{equation}\label{eq: splitting the FIO on balls3}
  \int_{\mathbb{R}^{n}}|(T_a^\varphi)^* \at(x)|\dd x=\int_{B'}|(T_a^\varphi)^* \at(x)|\dd x+\int_{\Rn \setminus B'}|(T_a^\varphi)^*\at (x)|\dd x,
\end{equation}
where $B'$ is the ball centered at the origin with radius $2K$ and
$$K:= \sup_{(y,\xi)\in \supp a\,\cap\, (\R^n \times \mathbb{S}^{n-1})}|\nabla_\xi \varphi(y,\xi)|.$$
We treat the first term of \eqref{eq: splitting the FIO on balls3} as in \eqref{eq:largeball}. For the second term we observe that the kernel of
$(T_a^\varphi)^*$ satisfies
\begin{equation}\label{based on nonstationary}
\abs{\int_{\R^n} e^{ix\cdot\xi -i\varphi(y,\xi)} \,a(y,\xi)\ddd \xi } \lesssim \frac 1{\abs x^{N}},
\end{equation}
for $|x|>2K$. This follows from the fact that, on the support of $a(y,\xi),$ the modulus of the gradient of the phase of the oscillatory integral above satisfies
$$|x-\nabla_\xi\varphi(y,\xi)|\geq |x|-K \geq |x|/2.$$
Now if
$$\psi_0(\xi)+\sum_{j=1}^\infty \psi (2^{-j}\xi)=1$$
is a Littlewood-Paley partition of unity with $\supp \psi$ inside a fixed annulus (see Definition \ref{def:LP}), then using Remark \ref{low freq rem} we have
\begin{equation}\label{eq:adjointlarger}
\begin{split}
\Big|\int_{\R^n} e^{ix\cdot\xi -i\varphi(y,\xi)}\, a(y,\xi)\ddd \xi \Big|
& \lesssim
\sum_{j=1}^{\infty}
\Big|\int_{\R^n} e^{ix\cdot\xi -i\varphi(y,\xi)}\, \psi(2^{-j}\xi)\, a(y,\xi)\ddd \xi \Big|\\
& = \sum_{j=1}^{\infty}  2^{jn}
\Big|\int_{\R^n} e^{i\lambda \Phi (x, y, \xi)}\,b(y,\xi)\ddd \xi \Big|,
\end{split}
\end{equation}

with $\lambda:= 2^j |x|$,
$$\Phi(x,y,\xi)
:=\frac{x\cdot\xi -\varphi(y,\xi)}{|x|},$$
and
$$b(y,\xi):=\psi (\xi)\, a(y,2^j\xi),$$ with compact support in $y$ and annulus-support in $\xi$. Now since for all multi-indices $\alpha$, $|\partial^{\alpha}_\xi b(y,\xi)|\lesssim 2^{jm}$ and since for $(y,\xi) \in \supp b(y,\xi)$ one has that $|\nabla_\xi\Phi(x,y,\xi)|\gtrsim 1$, the non-stationary phase estimate of Lemma \ref{lem:non-stationary} could be used to deduce that
$$\abs{\int_{\R^n} e^{i\lambda \Phi (x, y, \xi)}\,b(y,\xi)\ddd \xi}\lesssim 2^{jm}\, (2^j |x|)^{-N},$$
for any $N>0$. Thus using this in \eqref{eq:adjointlarger} and summing in $j$, \eqref{based on nonstationary} follows. Hence
\m{
\int_{\Rn \setminus B'}|(T_a^\varphi)^*\,\at(x)|\dd x \lesssim \int_{\Rn \setminus B'}\frac 1{\abs x^{N}}\brkt{\int_B \abs{\at(y)}\dd y} \dd x \lesssim 1.
}\hspace*{1cm}\\

\textbf{Step 3 - Globalisation of Steps 1 \& 2 }

Now we globalise the result that we have obtained so far for both $T_a^\varphi$ and $(T_a^\varphi)^* $ at the same time. Whenever we write $T$ we refer to both $T_a^\varphi$ and $(T_a^\varphi)^* $.\\

To prove that
$$ \int_{\mathbb{R}^{n}}|T\at(x)|\dd x
\lesssim 1$$
when there is no requirement on the support of the amplitude, we need to use a different strategy. First we observe that a global norm estimate for $T\at$ with $\at$ supported in a ball with an arbitrary centre, would follow from a  norm-estimate that is uniform in $s$ for $\tau_s^* T\tau_s \at$, with an atom $\at$ whose support is inside a ball centred at the origin. Note that here $\tau_s$ is the operator of translation by $s\in\R^n$. This is because by translation invariance of the $L^1$-norm one has that
$$\Vert T\at\Vert_{L^1(\R^n)}= \Vert \tau_s^* T\tau_s \tau_{-s} \at\Vert_{L^1(\R^n)}.$$  Thus our goal is to establish that
$$\Vert \tau_s^* T\tau_s \at\Vert_{L^1(\R^n)}\lesssim 1,$$
where the estimate is uniform in $s$ and $\at$ has its support in a ball centred at the origin.\\

At this point we once again use the conditions on the phase function to reduce our analysis to the case of operators with $\varphi$ of the form
$\theta (x,\xi)+x\cdot \xi$ or $-\theta (y,\xi)-y\cdot \xi$
with $\theta \in \Phi^1$, which can be done by the discussions of Section \ref{RS globalisation}. Now let $r\geq1$, $ L>n/\rho$ and $s\in\R^n$ and
suppose $\at$ is an $\mathscr{H}^1$-atom supported in a ball $B$, centred at the origin, with radius $r$. We use the the notions that were introduced in connection to the globalisation procedure in Section \ref{RS globalisation} and split the $L^1$-norm of $\tau_s^* T\tau_s \at$ into following two pieces:
\[
\left\| \tau_s^* T\tau_s \at \right\|_{L^1(\R^n)}
=
\left\|\tau_s^* T\tau_s \at\right\|_{L^1(\widetilde{\Delta}_{2r})}
+\left\|\tau_s^* T\tau_s \at\right\|_{L^1(\R^n \setminus \widetilde{\Delta}_{2r})}.
\]
First let us show that
\[
\left\|\tau_s^* T\tau_s \at\right\|_{L^1(\widetilde{\Delta}_{2r})}
\leq  C(n,M_{L},N_{L+1}).
\]

By Lemma \ref{Lem:outside}, for $x\in \widetilde{\Delta}_{2r}$ and $|y|\leq r$, we have
$$\widetilde{H}(x)\leq2 H(x,y,x-y)$$
and $(x,y,x-y)\in\Delta_r$.
Letting $K(x, y,x-y)$ denote the integral kernel of operator $T$, this fact and Lemma \ref{Lem:kernel RuzhSug} yield for any atom $\at$ supported in $B(0,r)$ that
\begin{align}\label{Ruz-Sug ball estim}
| T \at(x)|
&\leq 2^{L}\widetilde{H}(x)^{-L}\int_{|y|\leq r}\abs{H(x,y, x-y)^{L}\,K(x,y,x-y)\, \at(y)} \dd y
\\\nonumber
&\leq 2^{L}\widetilde{H}(x)^{-L}\,
\| H^{L} K\|_{L^\infty( \Delta_r)}\,
\Vert \at \Vert_{L^1(\R^n)}\\\nonumber
&\leq C(n,L,M_L ,N_{L+1})\,  \widetilde{H}(x)^{-L}\,,
\end{align}
since $ \Vert \at\Vert_{L^1(\R^n)} \leq 1.$
Therefore, if $r\geq 1$, choosing $ L>  n/\rho$, Lemma \ref{Lem:kernel RuzhSug} and the monotonicity
of $\Delta_r$ yield
\begin{align*}
\left\|T \at\right\|_{L^1(\widetilde{\Delta}_{2r})}
&\lesssim
\|\widetilde{H}(x)^{-L}\|_{L^1(\widetilde{\Delta}_{2r})}
\leq C(n, M_L, N_{L+1}).
\end{align*}
Observe that the phase function and the amplitude of $\tau_s^* T\tau_s$ are of the form  {$ \theta (x+ \nolinebreak  s, \xi)+(x-y)\cdot\xi$ and  $\sigma (x+s, \xi)$} respectively when $T= T_a^\varphi$ (a similar property is also true for $(T_a^\varphi)^*$). Therefore the conjugation of $T$ by $\tau_s$ renders the constants $M_L$ and $N_{L+1}$ unchanged and therefore the estimate above also yields the very same one for $\tau_s^* T\tau_s$. This means that
$$\left\|\tau_s^* T\tau_s \at\right\|_{L^1(\widetilde{\Delta}_{2r})}
\lesssim 1.$$

On the other hand for
$\left\|\tau_s^* T\tau_s \at\right\|_{L^1
(\R^n \setminus \widetilde{\Delta}_{2r})},$
Lemma \ref{Lem:outside}, H\"older's inequality and the properties of the atom $\at$ yield that
\begin{align*}
\left\|\tau_s^* T\tau_s \at\right\|_{L^1(\R^n \setminus \widetilde{\Delta}_{2r})}
&\leq
|\R^n \setminus \widetilde{\Delta}_{2r}|^{1/2}\,
\left\|\tau_s^* T\tau_s \at \right\|_{L^2(\R^n)}
\\
&\lesssim
r^{n/2}\,\|\at\|_{L^2(\R^n)}
\lesssim r^{n/2}\, r^{-n/2 } = 1.
\end{align*}
Now if the atom is supported in a ball of radius $r\leq 1$ then clearly $\supp \at\subset B(0,1).$ Now write $\R^n = \widetilde{\Delta}_2 \cup (\R^n \setminus \widetilde{\Delta}_2)$ and observe that we can now use Lemma \ref{Lem:kernel RuzhSug} with $r=1$ to conclude that
\begin{equation*}
|T\at(x)| \lesssim\widetilde{H}(x)^{-L},
\end{equation*}
which in turn yields that
$$\Vert \tau_s^* T\tau_s \at\Vert_{L^1 (\widetilde{\Delta}_2)}\lesssim 1.$$
Using now the first part of Lemma \ref{Lem:outside} we see that $\R^n \setminus \widetilde{\Delta}_2 \subset B(0, 2+N_{K})$ which together with the local boundedness result that we established previously implies that
\begin{equation*}
\Vert \tau_s^* T\tau_s \at\Vert_{L^1(\R^n \setminus\widetilde{\Delta}_2)}
\lesssim \Vert \tau_s^* T\tau_s \at\Vert_{L^1 (B(0, 2+N_{K}))}\lesssim \Vert \at\Vert_{\mathscr{H}^1 (\R^n)}\lesssim 1.
\end{equation*}
Now that we have boundedness from $\mathscr{H}^1(\Rn)$ to $L^1(\Rn)$ for both $T_a^\varphi$ itself and its adjoint as well as $L^2$-boundedness we can use a standard Riesz-Thorin interpolation argument to conclude that $T_a^\varphi$ is bounded from $L^p(\R^n)$ to itself.
\end{proof}

\subsection{Forbidden amplitudes}

The case of operators with amplitudes in $S^m_{\rho, 1}(\Rn)$ with $0\leq \rho \leq 1$ is rather special since the FIOs in question are generically not $L^2$-bounded. However Proposition \ref{basicL2} yields that if $m< n(\rho-1)/2$ then the associated FIO is indeed $L^2$-bounded, and this result is sharp. Here, only for the sake of completeness of exposition we state the result proven in \cite{DS} regarding the $L^p$-boundedness of FIOs with forbidden amplitudes.

\begin{Th}\label{thm:LpResultforbidden}
Let $n\geq1$, $a\in S^m_{\rho,1}(\Rn)$, $\varphi$ be an \emph{SND} phase function in the class $\Phi^2$ and let $T_a^\varphi$ be given as in \emph{Definition \ref{def:FIO}}. For  $0 \leq \rho\leq  1$ and
$$m<
n(\rho-1)\max \Big(\frac{1}{p}, \frac{1}{2}\Big)+ (n-1)
\Big|\frac{1}{p}-\frac{1}{2}\Big|$$
the \emph{FIO} $T_a^\varphi$ is $L^p$-bounded for $1\leq p\leq \infty.$
\end{Th}
\begin{proof}
See \cite[Propositions 2.3 and 2.5]{KS} for the case $n=1$, which is essentially the pseudodifferential case, and \cite[Theorem 2.17]{DS} for $n\geq 2$.
\end{proof}

\section{Sobolev space boundedness of FIOs with \texorpdfstring{$S^{0}_{1,1}$}\ -amplitudes}\label{Sobolev bounedness}

It turns out that just as in the case of pseudodifferential operators, the FIOs with forbidden amplitudes, say in $S^0_{1,1}(\Rn)$, despite failing to be $L^2$-bounded are bounded on $H^s(\Rn)$ with $s>0$. As was mentioned in the introduction, the proof of the Sobolev-boundedness in the pseudodifferential case goes back to E. Stein and independently to Y. Meyer. Other proofs were given by Bourdaud {\cite{Bourd}} and H\"ormander {\cite{Hormander2}}. Following Bourdaud, we establish the Sobolev boundedness of FIOs with amplitudes in the class $S^{0}_{1,1}(\R^n)$, as a consequence of the following more general result.

\begin{Th}\label{thm:sobolev}
Let $n\geq 1$, $a\in C^{r}_{*}S^0_{1,1}(\Rn)$ for some {$r>0 $} and $\varphi$ be an $\mathrm{SND}$ phase function in the class $\Phi^2$. Then for $0<s<r$ the \emph{FIO} $T_a^\varphi$ is bounded from the Sobolev space $H^{s}(\R^n)$ to $H^{s}(\R^n)$.
\end{Th}

\begin{proof}
We divide the proof into steps. \\

\textbf{Step 1 - Reduction of the FIOs with amplitudes in  $\mathbf{C^{r}_{*}S^0_{1,1}(\Rn)}$ class}\\
As was done in \cite{Bourd, Meyer}, it is enough to show the result for elementary amplitudes in the class $C_{*}^{r} S_{1, 1}^{0}(\Rn)$ where $r$ can be taken as any arbitrary positive number.

By definition, an elementary symbol in $C_{*}^{r} S_{1, 1}^{0}(\Rn)$ is of the form
\begin{equation}\label{decomposition of the amplitude}
a(x, \xi)=\sum_{k=0}^{\infty} M_{k}(x)\, \psi_{k}(\xi),
\end{equation}
 where $\psi_{k}$ was introduced in Definition \ref{def:LP}
 and $M_{k}(x)$ satisfies
\begin{equation}\label{estimates for Qk}
\left|M_{k}(x)\right|\lesssim 1,\qquad\left\|M_{k}\right\|_{C_{*}^{r}(\Rn)} \lesssim  2^{k r},
\end{equation}
where the $C_{*}^{r}$-norm is given in Definition \ref{def:Zygmund}. \\

We treat the case $k=0$ (the low frequency portion of the FIO) separately, so for now assume that $k\geq 1$.\\

Using the Littlewood-Paley partition of unity $I=\sum_{j=0}^\infty \psi_{j} (D)$ and setting
$$M_{kj}(x):= \psi_j (D) M_k(x),
\quad
f_k:= \psi_k(D) f
\quad \text{and} \quad
F_k:= T_{1}^\varphi f_k,$$
(the amplitude of the FIO $T_{1}^\varphi$ is identically equal to one)  we have that for $k\geq 1$
\begin{equation}\label{estimates for Qk and Fk}
\Vert M_{kj}\Vert_{L^\infty(\R^n)}\lesssim 2^{r(k-j)}
\quad \text{and} \quad
\left\|F_k\right\|_{L^{2}(\R^n)} \lesssim  \Vert f_k\Vert_{L^2(\R^n)}.
\end{equation}

\quad\\
The first estimate in \eqref{estimates for Qk and Fk} can be shown using the bound $\left\|M_{k}\right\|_{C_{*}^{r}(\Rn)} \lesssim  2^{k r}$, relations \eqref{Zygmundegenskap} and that $$\int_{\Rn} 2^{jn}\, \psi^\vee\left(\frac{y}{2^{-j}}\right)  \dd y= \int_{\Rn} \psi^\vee\left(y\right) \dd y = 0,$$
as follows
\begin{align*}
|\psi_j(D)(M_k)(x)|
& \lesssim\Big|\int_{\Rn} 2^{jn}\, \psi^\vee\Big(\frac{y}{2^{-j}}\Big)\,
\Big(M_k(x-y)-M_k (x)\Big) \dd y\Big| \\
& \lesssim {2^{-jr+kr}} \int_{\Rn}\Big|2^{jn} \psi^\vee \Big(\frac{y}{2^{-j}}\Big)\Big|\, \frac{|y|^r}{2^{-jr}} \dd y \lesssim 2^{(k-j)r},
\end{align*}
for $j>0$. For $j=0$ this is a consequence of the $L^\infty$-boundedness of $\psi_0(D)$ and the first estimate in \eqref{estimates for Qk}.\\

The second estimate in \eqref{estimates for Qk and Fk} is of course a direct consequence of the $L^2$-boundedness of FIOs with amplitudes in $S^{0}_{1,0}(\R^n).$\\

Using the above notation we can now decompose $T^{\varphi}_{a}$ as
\begin{align*}
    T^{\varphi}_{a} f(x)= \sum_{k=1}^\infty    M_k(x)\, T_{1}^\varphi f_k(x)=\sum_{k=1}^\infty M_k(x)\,  F_k (x).
\end{align*}
At this point, taking into account the properties of the SND phase function $\varphi\in \Phi^2$ (i.e. $|\nabla_x \varphi(x, \xi)|\approx |\xi|$), Proposition \ref{thm:monsteriosity} and choosing a suitable annulus supported $\tilde{\psi}$ we have
\begin{align*}
F_k(x)
& = \int_{\Rn} \tilde{\psi}_k(\nabla_x\varphi(x,\xi))\, e^{i\varphi(x,\xi)} \,\widehat{f_k}(\xi) \ddd \xi \\
\nonumber
& = \tilde{\psi}_k(D)F_k(x)
- \sum_{{0<|\alpha|< N_1}} \frac{2^{-k\eps|\alpha|}}{\alpha!}\,T^{\varphi}_{\sigma_{\alpha,k}} f_k(x)
- 2^{-k\eps N_1}\,T^\varphi_{r_k} f_k(x) \\
\nonumber & =:  F^{1}_k(x) + F^{2}_k(x)+ F^{3}_k(x),
\end{align*}
with
\begin{equation*}
\left |\partial^{\beta}_{\xi} \partial^{\gamma}_{x} \sigma_{\alpha,k}(x,\xi) \right |
\lesssim \bra{\xi}^{-(1/2-\eps)|\alpha|-\abs\beta}, \qquad |\alpha|\geq 0,
\end{equation*}
\begin{equation*}
\supp_\xi \sigma_{\alpha,k}(x,\xi)
= \left \{ \xi\in\R^n: \ C_1 2^k\leq |\xi| \leq C_2 2^k \right \},
\end{equation*}
and
\begin{equation*}
\abs{\partial^{\beta}_{\xi} \partial^{\gamma}_{x}  r_k (x,\xi)}
\lesssim  \,\bra{\xi}^{-(1/2-\eps)N_1-\abs\beta},
 \end{equation*}
where the estimates above are uniform in $k$.

\quad \\
Thus
\begin{equation*}
     T^{\varphi}_{a} f(x)= \sum_{k=1}^\infty M_k(x)\,  F^{1}_k (x)+\sum_{k=1}^\infty M_k(x)\,  F^{2}_k (x)+\sum_{k=1}^\infty M_k(x)\,  F^{3}_k (x) .
\end{equation*}

\quad \\

\textbf{Step 2 - Analysis of $\mathbf{\sum_{k=1}^\infty M_k(x)\,  F^{1}_k (x)}$}\\
Now to analyse $F^{1}_k$ we write
$$M_k = \sum_{j=0}^\infty \psi_j(D) M_k =:\sum_{j=0}^\infty  M_{kj}$$
and split the sum in $j$ into the following pieces
\m{
\sum_{k=1}^\infty \sum_{j=0}^\infty  M_{kj}(x)\, F_k^1 (x)= \sum_{k=1}^{\infty} \sum_{j=0}^{k-1}  M_{kj}(x)\, F_k^1 (x) +\sum_{k=1}^{\infty} \sum_{j=k}^{\infty}  M_{kj}(x)\, F_k^1(x) =: \bf{A} +\bf{B}.
}

Firstly, we establish the $H^s$-boundedness of $\bf{A}$. To this end we have
$$ \sum_{k=1}^{\infty} \sum_{j=0}^{k-1}  M_{kj}(x)\, F^{1}_{k} (x)=:\sum_{k=1}^{\infty }b_k(x). $$

The Fourier transform of $b_k$ is given by
\begin{equation*}
    \widehat{b_k}(\eta)= \sum_{j=0}^{k-1}\int_{\Rn} \psi_j(\eta-\xi)\,\widehat{ M_k}(\eta-\xi)\,    \tilde{\psi_k}(\xi)\,\widehat{F^1_k}(\xi)  \ddd \xi.
\end{equation*}
From this and what we know about the support of convolutions, it follows that spectrum of $b_k$ is contained in an annulus $|\eta|\approx 2^k$. From this and Lemma \ref{lem:generalisedTL} it follows that for $s>0$,
\begin{align*}\label{sobnorm of A}
\norm{ \sum_{k=1}^{\infty }b_k }_{H^s(\R^n)}
& \lesssim
\Big\|\Big\{\sum_{k=1}^{\infty} 4^{k s}\,\Big|\sum_{j=0}^{k-1} M_{k j}\, F_{k}^1\Big|^{2}\Big\}^{1 / 2}\Big\|_{L^{2}(\R^n)}\\
& \lesssim
\Big\|\Big\{\sum_{k=1}^{\infty} 4^{k s}\,\|M_{k}\|_{L^{\infty}(\R^n)}^{2}\,|F_{k}^1|^{2}\Big\}^{1 / 2}\Big\|_{L^{2}(\R^n)} \\
& \lesssim
\Big\|\Big\{\sum_{k=1}^{\infty} 4^{k s}\,|F^1_{k}|^{2}\Big\}^{1/2}\Big\|_{L^{2}(\R^n)}
 = \Big\{\sum_{k=1}^{\infty} 4^{k s}\,\Vert F_{k}^1\Vert_{L^2(\R^n)}^{2}\Big\}^{1/2}\\
& \lesssim
\Big\{\sum_{k=1}^{\infty} 4^{k s}\,\Vert f_{k}\Vert_{L^2 (\R^n)}^{2}\Big\}^{1/2}
\lesssim
\Vert f\Vert_{H^{s} (\R^n)},
\end{align*}
where we have used Remark \ref{rem:sobolevnormer} and that
$$ \Big|\sum_{j=0}^{k-1} M_{kj}\Big|\lesssim \Vert M_k\Vert_{L^\infty(\Rn)}\lesssim  1.$$

Now for term $\mathbf{B}$ in $F^1_k$ applying Lemma \ref{lem:generalisedTL} to
$$h_j(x):=\sum_{k=1}^{j}  M_{kj}(x)\, F_k^1(x),$$
we obtain (using Fubini's theorem for sums, \eqref{estimates for Qk and Fk} and Young's inequality for discrete convolutions)
\begin{align*}
&\norm{ \sum_{k=1}^{\infty} \sum_{j=k}^{\infty} M_{kj} \,F_k^1}_{H^s(\Rn)}
 \lesssim \norm{ \sum_{j=1}^{\infty} \sum_{k=1}^{j} M_{kj} \,F_k^1}_{H^s(\Rn)}
\\
&\qquad=\norm{ \sum_{j=1}^{\infty} h_j}_{H^s(\Rn)}  \lesssim \norm{ \set{\sum_{j=1}^\infty 4^{js}\,\abs{h_j}^2}^{1/2}}_{L^2(\Rn)}
\\
&\qquad\lesssim \norm{ \set{\sum_{j=1}^\infty 4^{js}\,\brkt{\sum_{k=1}^{j} 2^{(k-j)r}\, \abs{ F_k^1}}^2}^{1/2}}_{L^2(\Rn)}
\\
&\qquad= \norm{\set{\sum_{j=1}^\infty \brkt{2^{(\cdot)(s-r)}*2^{(\cdot)s}F_{(\cdot)}^1(j)}^2}^{1/2} }_{L^2(\Rn)}
\\
&\qquad\lesssim \norm{ \set{\sum_{j=1}^\infty2^{j(s-r)}}\set{\sum_{j=1}^\infty \brkt{2^{js}\,\tilde{\psi}_{j}(D)F_{j}}^2}^{1/2}}_{L^2(\Rn)}
\\
&\qquad \lesssim \|f\|_{H^s(\Rn)},
\end{align*}
because of Remarks \ref{rem:Sobolev}, \ref{rem:sobolevnormer} and $r>s$.\\

\textbf{Step 3 - Analysis of $\mathbf{\sum_{k=1}^\infty M_k(x)\,  F^{2}_k (x)}$}\\
To establish the Sobolev boundedness for the term $\sum_{k=1}^\infty M_k(x)\,  F^{2}_k (x)$,
we would like to understand the action of the Littlewood-Paley operator $\psi_j(D)$ on  this term in order to use Definition \ref{def:Sobolev} together with Remark \ref{rem:sobolevnormer}.\\

Then for some integer $N_2>s$ and $0<\eps'<1/2$, write
\begin{equation}\label{FK2sobolev}
\begin{split}
\psi_j(D) \left( M_k(x)\, F^2_{k}\right)
 = M_k(x)\, \psi_j(D)F^2_k(x)+ [ \psi_j(D), M_k] F^2_{k}(x)
=\\   M_k(x)\, \sum_{{0<|\alpha|< N_1}}  \frac{2^{-k\eps|\alpha|}}{\alpha!}\Big(\sum_{{|\beta|<N_2}} \frac{2^{-j\eps'|\beta|}}{\beta!}
T^{\varphi}_{\sigma_{\alpha, \beta,k, j}}
+ 2^{-j\eps' N_2}T^\varphi_{r_{j,k}}\Big) f_k(x)\\+ [ \psi_j(D), M_k] F_{k}^2 
  =: \mathrm{I}+\mathrm{II}+\mathrm{III}
\end{split}
\end{equation}
 with
\begin{equation*}
\left |\partial^{\delta}_{\xi} \partial^{\gamma}_{x} \sigma_{\alpha,\beta, k,j}(x,\xi) \right |
\lesssim \bra{\xi}^{-(1/2-\eps)|\alpha|-(1/2-\eps')|\beta|-\abs\delta}, \qquad |\alpha|> 0, \quad |\beta|\geq 0,
\end{equation*}
\begin{equation*}
\supp_\xi \sigma_{\alpha,\beta,k,j}(x,\xi)
= \left \{ \xi\in\R^n: \ C_1 2^j\leq |\xi| \leq C_2 2^j \right \},
\end{equation*}
and
\begin{equation*}
\abs{\partial^{\delta}_{\xi} \partial^{\gamma}_{x}  r_{j,k}(x,\xi)}
\lesssim  \,\bra{\xi}^{-(1/2-\eps)N_1-(1/2-\eps')N_2-\abs\delta},
\end{equation*}
where both estimates above are uniform in $j$ and $k$.
Moreover, {since in the decomposition \eqref {decomposition of the amplitude} of $a(x,\xi)$, we are at present considering the parts supported outside a neighbourhood of the origin in the $\xi$-variable}, i.e. those for which $k\geq 1$, we also have that $r_{j,k}(x,\xi)$ also vanishes in a neighbourhood of $\xi=0$.\\

For term $\mathrm{I}$, and in light of the support properties of $\sigma_{\alpha,\beta,k,j}$, we claim that (uniformly in $j$ and $k$)
\begin{equation} \label{eq:ttstar argument}
\begin{split}
\|T^\varphi_{\sigma_{\alpha,\beta,k,j}} f\|_ {L^2(\R^n)}
\lesssim \|\Psi_j(D)f\|_{L^2(\R^n)},
\end{split}
\end{equation}
where $\Psi_j$ is a Littlewood-Paley-type frequency localisation that is equal to one on the support of $\sigma_{\alpha,\beta,k,j}.$ Therefore $T^\varphi_{\sigma_{\alpha,\beta,k,j}} f= T^\varphi_{\sigma_{\alpha,\beta,k,j}} \Psi_j(D) f,$ and it is enough to show that
\begin{equation} \label{eq:ttstar argument2}
\begin{split}
\|T^\varphi_{\sigma_{\alpha,\beta,k,j}} f\|_ {L^2(\R^n)}
\lesssim \|f\|_{L^2(\R^n)},
\end{split}
\end{equation}
uniformly in $k$ and $j$.
To see this, we proceed by studying the boundedness of $S_j:=T^\varphi_{\sigma_{\alpha,\beta,k,j}}(T^\varphi_{\sigma_{\alpha,\beta,k,j}})^{\ast}$.  A simple calculation shows that
$$ S_j f(x)=\int_{\R^n} K_j(x,y)\, f(y)\dd y,$$
with
 \begin{equation*}
 K_j(x,y):=\int_{\R^n} e^{i\varphi(x,\xi)-i\varphi(y,\xi)}\,  \,\sigma_{\alpha,\beta,k,j}(x,\xi)\, \overline{\sigma_{\alpha,\beta,k,j}(y,\xi)}\ddd\xi .
 \end{equation*}
Now since $\varphi$ is homogeneous of degree one in the $\xi$ variable, $K_j (x,y)$ can be written as
\begin{equation*}
K_{j}(x,y)=2^{j n}\int_{\R^n} b_{j }(x,y,2^{j }\xi)\,
e^{i2^j  \Phi(x,y,\xi)}\ddd\xi,
\end{equation*}
with
$$\Phi(x,y,\xi)
:= \varphi (x,\xi) -\varphi (y,\xi),$$
and
$$ b_j (x,y,\xi )
:=\sigma_{\alpha,\beta,k,j}(x,\xi)\, \overline{\sigma_{\alpha,\beta,k,j}(y,\xi)}.$$
Observe that the $\xi$-support of $b_{j }(x,y,2^{j }\xi)$ lies in the compact set $\mathcal{K}:=\left \{C_1\leq \abs{\xi}\leq C_2\right \}$.  From the SND condition \eqref{eq:SND} it also follows that
\begin{equation}\label{Phi cond 1}
\vert \nabla_{\xi}\Phi (x,y, \xi)\vert \approx \vert x-y\vert, \quad \text{for any $x,y\in \R^n$ and $\xi\in \mathcal{K}$}.
\end{equation}
Assume that $N_3>n$ is an integer, fix $x\neq y$ and set $\phi(\xi):=\Phi (x,y, \xi)$, $\vartheta:=\abs{\nabla_\xi \phi}^2$. By the mean value theorem, \eqref{C_alpha} and \eqref{Phi cond 1}, for any multi-index $\alpha$ with $\abs{\alpha}\geq 1$ and any $\xi\in \mathcal{K}$,
\[
    \abs{\d^\alpha_{\xi} \phi(\xi)}\lesssim \vert \nabla_{\xi}\Phi (x,y, \xi)\vert= \vartheta^{1/2}.
\]
On the other hand, since
$$ \d^\alpha_{\xi} \vartheta =\sum_{\nu=1}^n \sum_{\beta\leq\alpha}  \binom{\alpha}{\beta} \d^\beta_{\xi}\d_{\xi_\nu} \phi\, \d^{\alpha-\beta}_{\xi}\d_{\xi_\nu}  \phi,$$
it follows that, for any $\abs{\alpha}\geq 0$, $\abs{\d^\alpha_{\xi} \vartheta}\lesssim \vartheta$. We estimate the kernel $K_j$ in two different ways. For the first estimate, \eqref{Phi cond 1} and Lemma \ref{lem:non-stationary} with $F(\xi):=b_{j } (x,y,2^{j }\xi  ),$ yield
\begin{align}\label{eq:RLSest1}
\vert K_{j }(x,y)\vert
&\leq \,2^{j  n} 2^{-j  N_3}\ C_{M,\mathcal{K}} \sum _{\vert \alpha \vert\leq N_3} 2^{j \abs{\alpha}}\int_{\R^n} {\left \vert \partial^{\alpha}_{\xi} b_{j }(x,y,2^j  \xi) \right \vert\, \big \vert \nabla_{\xi}\Phi(x,y,\xi) \big \vert^{-N_3}} \ddd\xi \nonumber \\
&\lesssim 2^{-j  N_3} \abs{x-y}^{-N_3} \sum _{\vert \alpha \vert\leq N_3}  2^{j \abs{\alpha}} \int_{\R^n} \abs{\partial^{\alpha}_{\xi} b_{j }(x,y,\xi)}\ddd\xi \\ &\lesssim  2^{j n}\,
\Big(2^{j  } \abs{x-y}\Big)^{-N_3}, \nonumber
\end{align}
where the fact that the $\xi$-support of $b_j $ lies in a ball of radius $\approx 2^j $ and that for $|\alpha|\geq 0$
\begin{equation}\label{eq:mk}
	\abs{\partial^{\alpha}_{\xi} b_{j }(x,y,\xi)} \lesssim 2^{-j|\alpha| },
\end{equation} have been used. By \eqref{eq:mk} we also obtain
\begin{equation}\label{eq:RLSest2}
\vert K_{j }(x,y)\vert \leq 2^{jn} \int_{\R^n} \abs{b_{j }(x,y,2^{j }\xi)}\ddd\xi \lesssim 2^{j n},
\end{equation}
and when combining estimates \eqref{eq:RLSest1} and \eqref{eq:RLSest2} one has
\begin{equation}\label{eq:Kk}
	\vert K_{j }(x,y)\vert\lesssim  2^{j n}\,\brkt{1+2^j \abs{x-y}}^{-N_3}.
\end{equation}
Thus, using \eqref{eq:Kk} and Minkowski's inequality we have
\m{
\|S_j  f\|_{L^2(\R^n)}
\lesssim \Big\| \int_{\R^n} 2^{j n}\,\brkt{1+2^j  |y|}^{-N_3}\,f(\,\cdot-y)\dd y \Big\|_{L^2(\R^n)}
\lesssim \|f\|_{L^2(\R^n)}. }

Now the Cauchy-Schwarz inequality yields
\begin{align*}
\|(T^\varphi_{\sigma_{\alpha,\beta,k,j}})^{\ast}f\|^2_{L^2(\R^n)}
& =  \Big\langle {T^\varphi_{\sigma_{\alpha,\beta,k,j}}} {(T^\varphi_{\sigma_{\alpha,\beta,k,j}})^{\ast}}f,f
\Big\rangle_{L^2(\R^n)}  \\
& \lesssim \|S_j f\|_{L^2(\R^n)}
\, \|f\|_{L^2(\R^n)} \\
& \lesssim \|f\|^2_{L^2(\R^n)}.
\end{align*}
Therefore
$$
\|{T^\varphi_{\sigma_{\alpha,\beta,k,j}}}\|_{L^2(\R^n)\to L^2(\R^n)}
=\|{(T^\varphi_{\sigma_{\alpha,\beta,k,j}})^*}\|_{L^2(\R^n)\to L^2(\R^n)}\lesssim 1,$$
and \eqref{eq:ttstar argument2} is proven.\\

Now \eqref{estimates for Qk}, Cauchy-Schwarz inequality, Fubini's theorem, \eqref{eq:ttstar argument}, and the definition of the Sobolev norm yield that
\begin{align}\label{skattning I}
& \sum_{j=0}^{\infty} 4^{js}\norm{  \sum_{k=1}^{\infty}M_k(x)\, \sum_{{0<|\alpha|< N_1}}  \frac{2^{-k\eps|\alpha|}}{\alpha!}\sum_{{|\beta|<N_2}} \frac{2^{-j\eps'|\beta|}}{\beta!}
T^{\varphi}_{\sigma_{\alpha, \beta,k, j}}f_k} ^2_{L^2(\Rn)}\\
&\nonumber \qquad \quad
 \lesssim   \sum_{k=1}^{\infty} 2^{-k\varepsilon}\sum_{j=0}^{\infty} 4^{js}
 \| \Psi_j(D) f_k \|^2_{L^2(\Rn)} \lesssim \sum_{k=1}^{\infty} 2^{-k\varepsilon}\|f_k\|^2_{H^s(\Rn)}\lesssim \|f\|^2_{H^s(\Rn)}.
\end{align}

\quad\\

For term $\mathrm{II}$, we decompose $T^{\varphi}_{r_{j,k}}$ into Littlewood-Paley pieces as follows:
\m{T^{\varphi}_{r_{j,k} } f(x) = \sum_{{\ell=0}}^\infty \int_{\R^n} e^{i\varphi(x,\xi)}\,r_{j,k} (x,\xi)\,\psi_\ell(\xi)\,\widehat{f}(\xi) \ddd\xi =: \sum_{\ell=0}^\infty T^{\varphi}_{r_{j,k,\ell}}f(x),}
where the $\psi_\ell$'s are defined in Definition \ref{def:LP}. By a proof identical to the one of \eqref{eq:ttstar argument}, we see that
\begin{equation*}
\|T^{\varphi}_{r_{j,k,\ell}}f\|_{L^2(\R^n)}
\lesssim
\|\Psi_\ell(D)f\|_{L^2(\R^n)}.
\end{equation*}
Thus
\begin{equation}\label{eq:highfreq6}
\begin{split}
\|T^{\varphi}_{r_{j,k}} f\|_{L^2(\R^n)}
&\lesssim \sum_{\ell =0}^\infty
\|T^{\varphi}_{r_{j,k,\ell }}f\|_{L^2(\R^n)}
\lesssim \sum_{\ell =0}^\infty
\|\Psi_\ell (D)f\|_{L^2(\R^n)}.
\end{split}
\end{equation}
Note that the estimate \eqref{eq:highfreq6} is uniform in $j$.
Then we claim that for $s>0$ one has
\begin{equation*}
T^{\varphi}_{r_{j,k}}:H^s(\R^n)\to L^2(\R^n).
\end{equation*}
Indeed the Cauchy-Schwarz inequality yields
\begin{equation*}
\begin{split}
\|T^{\varphi}_{r_{j,k}} f\|_{L^2(\R^n)}
& \lesssim \sum_{\ell =0}^{\infty }
\|\Psi_\ell (D)f\|_{L^2(\R^n)}
= \sum_{\ell =0}^\infty 2^{-\ell s }\,\brkt{2^{\ell s}\,\|\Psi_\ell (D)f\|_{L^2(\R^n)}}  \\
& \lesssim \Big( \sum_{\ell =0}^\infty 2^{-2\ell
s}\Big)^{1/2}
\Big( \sum_{\ell =0}^\infty 2^{2\ell s}\,\|\Psi_\ell (D)f\|^{2}_{L^2(\R^n)}\Big)^{1/2}
\lesssim \|f\|_{H^{s}(\R^n)}.
\end{split}
\end{equation*}
The last step follows from the definition of $H^s(\R^n)$, see Remark \ref{rem:Sobolev}.
Thus, for $N_2$ large enough,
\begin{align}\label{skattning II}
& \sum_{j=0}^{\infty} 4^{js}\norm{ \sum_{k=1}^{\infty} M_k 2^{-j\eps' N_2} \sum_{{0<|\alpha|< N_1}}  \frac{2^{-k\eps|\alpha|}}{\alpha!}T^\varphi_{r_{j,k}}f_k} ^2_{L^2(\Rn)}\\
&\nonumber \qquad \quad
 \lesssim \Big(\sum_{k=1}^{\infty} 2^{-k\varepsilon}\Big)^2  \sum_{j=0}^{\infty}
4^{j(s-\varepsilon'N_2)}\Vert f\Vert^2_{H^s(\Rn)} \lesssim  \|f\|^2_{H^s(\Rn)}.
\end{align}

\quad\\

For term $\mathrm{III}$ of  \eqref{FK2sobolev}, the second estimate in \eqref{estimates for Qk} and relations \eqref{Zygmundegenskap} yield that
\nm{eq:commutator}{
\left|\left[\psi_j(D), M_k\right](F^{2}_{k})(x)\right|
& \lesssim
 \Big|\int_{\Rn} 2^{jn}\, \psi^\vee\left(\frac{x-y}{2^{-j}}\right)\,(M_k(y)-M_k (x)) \,F^{2}_{k} (y) \dd y\Big|
 \\
 & \lesssim {2^{-jr+kr}} \int_{\Rn}\Big|2^{jn}\, \psi^\vee \Big(\frac{x-y}{2^{-j}}\Big)\Big|\, \frac{|x-y|^r}{2^{-jr}}\,|F^{2}_{k} (y)| \dd y
 \\
 & \lesssim {2^{-jr+kr}} \,\mathcal{M}(F^{2}_{k})(x),
}
for $j>0$, where $\mathcal{M}$ is the Hardy-Littlewood maximal function. For $j=0$ this follows from the $L^\infty$-boundedness of $\psi_0(D)$ and the first estimate of \eqref{estimates for Qk}.\\
Therefore using the
Proposition \ref{basicL2} on $T^{\varphi}_{\sigma_{\alpha,k}}$  and the commutator estimate above we obtain
\m{
\Big\| \sum_{k=1}^{\infty} [\psi_j(D), M_k] F^{2}_{k}\Big\|_{L^2(\R^n)}^2
& \lesssim 4^{-jr}\sum_{0<|\alpha|<N_1} \Big(\sum_{k=1}^{\infty} 2^{k(r-\varepsilon)}\,
\Vert T^{\varphi}_{\sigma_{\alpha,k}} f_k\Vert_{L^2(\R^n)}\Big)^2 \\
& \lesssim 4^{-jr}\, \Big(\sum_{k=1}^{\infty} 2^{k(r-\varepsilon)} \,
    \Vert  f_k\Vert_{L^2(\R^n)}\Big)^2.
}
Hence if $r\in (s, s+\varepsilon)$ and
using Cauchy-Schwarz's inequality we obtain
\begin{align*}
  \sum_{j=0}^\infty 4^{js}\,
  \Big\| \sum_{k=1}^{\infty} [\psi_j(D), M_k] F^{2}_{k}\Big\|_{L^2(\R^n)}^2
  & \leq \sum_{j=0}^\infty   4^{j(s-r)} \,
  \Big(\sum_{k=1}^{\infty} 2^{k(r-\varepsilon)} \, 2^{-ks}\,2^{ks} \, \Vert f_k\Vert_{L^2(\R^n)}\Big)^2   \\
  & \leq \sum_{j=0}^\infty   4^{j(s-r)}\, \Big(\sum_{k=1}^{\infty} 4^{k(r-\varepsilon -s)}\Big)\Big(\sum_{k=1}^{\infty} 4^{ks}\,\Vert f_k\Vert_{L^2(\R^n)}^2\Big)\\
  & \approx \Vert f\Vert^2_{H^s (\R^n)}.
  \end{align*}

  Thus putting this, \eqref{skattning I} and \eqref{skattning II} together we obtain
\begin{equation*}
    \norm{ \sum_{k=1}^{\infty} M_{k} \,F_k^2}_{H^s(\Rn)}\lesssim \Vert f\Vert_{H^s(\Rn)}.
\end{equation*}

\quad\\

\textbf{Step 4 - Analysis of $\mathbf{\sum_{k=1}^\infty M_k(x)\,  F^{3}_k (x)}$}\\
Finally we turn to the Sobolev boundedness of the term
\begin{equation*}
 \sum_{{k=1}}^{\infty} M_{k}(x)\, F^3_k (x).
 \end{equation*}
Once again, using the definition of $F^3_k$ above, we have
\m{\psi_j(D)  M_k(x) F^3_{k}(x)
     & = 2^{-k\eps N_1}  M_k(x)\, \psi_j(D)T^\varphi_{r_k} f_k(x)+2^{-k\eps N_1} [ \psi_j(D), M_k] T^\varphi_{r_k} f_k(x).}
The commutator term can be treated as term $\mathrm{III}$ above since we can choose arbitrarily large decay in $k$. After taking the sum in $k$ of the first term, using \eqref{estimates for Qk}, then taking the $L^2$-norm, multiplying with $4^{js}$ and then taking the $\ell^2$-norm in $j$ one has the estimate

\m{
&\sum_{k=1}^\infty 2^{-k\eps N_1}\,
\brkt{\sum_{j=0}^\infty \Big\| 2^{js}\, M_k(x)\, \psi_j(D)T^\varphi_{r_k} f_k(x)(x)\Big\|^2_{L^2(\Rn)}}^{1/2}\\ &\qquad\lesssim \sum_{k=1}^\infty 2^{-k\eps N_1}\,\brkt{\sum_{j=0}^\infty\norm{ 2^{js}\,  \psi_j(D)T^\varphi_{r_k} f_k(x)(x)}^2_{L^2(\Rn)}}^{1/2} \\
&\qquad\lesssim
\sum_{k=1}^\infty 2^{-k\eps N_1}\,\| T^\varphi_{r_k} f_k(x)\|_{H^s(\Rn)} {\lesssim }
\sum_{k=1}^\infty 2^{-k\eps N_1}\,\|f\|_{H^s(\Rn)}
\lesssim \|f\|_{H^s(\Rn)},
}
where we have also used Lemma \ref{sobolevtheorem}.\\

\textbf{Step 5 - Analysis of the case $k=0$}\\
We have, using the second estimate in \eqref{estimates for Qk}, a similar argument as in \eqref{eq:commutator}, and the Sobolev-boundedness result in Lemma \ref{sobolevtheorem}
\begin{align*}
\brkt{\sum_{j=0}^\infty 4^{js}\, \norm{\psi_j(D) M_0T_1^\varphi f_0 }_{L^2(\Rn)}^2}^{1/2}
& \leq \brkt{\sum_{j=0}^\infty 4^{js}\, \norm{ M_0\,\psi_j(D)T_1^\varphi f_0 }_{L^2(\Rn)}^2}^{1/2} \\
&\qquad +\brkt{\sum_{j=0}^\infty 4^{js}\, \norm{[\psi_j(D), M_0]T_1^\varphi f_0 }_{L^2(\Rn)}^2}^{1/2} \\
& \lesssim \|T_1^\varphi f_0\|_{H^s(\Rn)}
+\brkt{\sum_{j=0}^\infty 4^{j(s-r)}\, \|T_1^\varphi f_0 \|_{L^2(\Rn)}^2}^{1/2} \\
& \lesssim \|f_0\|_{H^s(\Rn)}
+ \|f_0\|_{L^2(\Rn)}
\lesssim \|f\|_{H^s(\Rn)}. \qedhere
\end{align*}

\end{proof}

\quad\\
Now since $S^m_{1,1} (\R^n)\subset C_{*}^{r} S_{1,1}^{m}(\R^n)$ for all $r>0$, we deduce the following:
\begin{Cor}

Let $a\in S^0_{1,1}(\Rn)$ and $\varphi\in\Phi^2$ be an $\mathrm{SND}$ phase function. Then the \emph{FIO} $T_a^\varphi$ is bounded from the Sobolev space $H^{s}(\R^n)$ to $H^{s}(\R^n)$, for all $s>0$.
\end{Cor}


\begin{thebibliography}{10}

\bibitem{AH}
{\sc J.~\'Alvarez and J.~Hounie}, {\em Estimates for the kernel and continuity
  properties of pseudo-differential operators}, Ark. Mat., 28 (1990),
  pp.~1--22.

\bibitem{Bony}
{\sc J.~M. Bony}, {\em Calcul symbolique et propagation des singularit\'{e}s
  pour les \'{e}quations aux d\'{e}riv\'{e}es partielles non lin\'{e}aires},
  Ann. Sci. \'{E}cole Norm. Sup. (4), 14 (1981), pp.~209--246.

\bibitem{Bourd}
{\sc G.~Bourdaud}, {\em Une alg\`ebre maximale d'op\'{e}rateurs
  pseudo-diff\'{e}rentiels}, Comm. Partial Differential Equations, 13 (1988),
  pp.~1059--1083.

\bibitem{CISY}
{\sc A.~J. Castro, A.~Israelsson, W.~Staubach, and M.~Yerlanov}, {\em
  Regularity properties of {S}chr\"odinger integral operators and general
  oscillatory integrals}.
\newblock \href{https://arxiv.org/abs/1912.08316} {(arXiv:1912.08316)}.

\bibitem{CM}
{\sc R.~R. Coifman and Y.~Meyer}, {\em Au del\`a des op\'{e}rateurs
  pseudo-diff\'{e}rentiels}, vol.~57 of Ast\'{e}risque, Soci\'{e}t\'{e}
  Math\'{e}matique de France, Paris, 1978.

\bibitem{CNR1}
{\sc E.~Cordero, F.~Nicola, and L.~Rodino}, {\em Boundedness of {F}ourier
  integral operators on {${\mathscr F}L^p$} spaces}, Trans. Amer. Math. Soc.,
  361 (2009), pp.~6049--6071.

\bibitem{CNR2}
\leavevmode\vrule height 2pt depth -1.6pt width 23pt, {\em On the global
  boundedness of {F}ourier integral operators}, Ann. Global Anal. Geom., 38
  (2010), pp.~373--398.

\bibitem{CR1}
{\sc S.~Coriasco and M.~Ruzhansky}, {\em On the boundedness of {F}ourier
  integral operators on {$L^p(\mathbf R^n)$}}, C. R. Math. Acad. Sci. Paris,
  348 (2010), pp.~847--851.

\bibitem{DS}
{\sc D.~Dos Santos~Ferreira and W.~Staubach}, {\em Global and local regularity
  of {F}ourier integral operators on weighted and unweighted spaces}, Mem.
  Amer. Math. Soc., 229 (2014), pp.~xiv+65.

\bibitem{Feff}
{\sc C.~Fefferman}, {\em A note on spherical summation multipliers}, Israel J.
  Math., 15 (1973), pp.~44--52.

\bibitem{HPR}
{\sc A.~Hassell, P.~Portal, and J.~Rozendaal}, {\em Off-singularity bounds and
  hardy spaces for fourier integral operators}.
\newblock To appear in Trans. Amer. Math. Soc. 2020, {arXiv:1802.05932v4}.

\bibitem{Hor1}
{\sc L.~H\"ormander}, {\em Pseudo-differential operators and hypoelliptic
  equations. \emph{Singular integrals: (Proc. Sympos. Pure Math., Vol. X)}},
  Amer. Math. Soc., Providence, R.I.,  (1967), pp.~138--183.

\bibitem{Hormander2}
\leavevmode\vrule height 2pt depth -1.6pt width 23pt, {\em Lectures on
  nonlinear hyperbolic differential equations}, vol.~26 of Math\'{e}matiques \&
  Applications (Berlin), Springer-Verlag, Berlin, 1997.

\bibitem{IRS}
{\sc A.~Israelsson, S.~Rodr\'iguez-L\'opez, and W.~Staubach}, {\em Local and
  global estimates for hyperbolic equations in {B}esov-{L}ipschitz and
  {T}riebel-{L}izorkin spaces}.
\newblock To appear in Analysis \& PDE 2020
  \href{https://arxiv.org/abs/1802.05932} {(arXiv:1802.05932v4)}.

\bibitem{KS}
{\sc C.~E. Kenig and W.~Staubach}, {\em {$\Psi$}-pseudodifferential operators
  and estimates for maximal oscillatory integrals}, Studia Math., 183 (2007),
  pp.~249--258.

\bibitem{MT}
{\sc R.~B. Melrose and M.~E. Taylor}, {\em The radiation pattern of a
  diffracted wave near the shadow boundary}, Comm. Partial Differential
  Equations, 11 (1986), pp.~599--672.

\bibitem{Meyer}
{\sc Y.~Meyer}, {\em R\'{e}gularit\'{e} des solutions des \'{e}quations aux
  d\'{e}riv\'{e}es partielles non lin\'{e}aires (d'apr\`es {J}.-{M}. {B}ony)},
  in Bourbaki {S}eminar, {V}ol. 1979/80, vol.~842 of Lecture Notes in Math.,
  Springer, Berlin-New York, 1981, pp.~293--302.

\bibitem{RRS}
{\sc S.~Rodr\'iguez-L\'opez, D.~Rule, and W.~Staubach}, {\em A
  {S}eeger-{S}ogge-{S}tein theorem for bilinear {F}ourier integral operators},
  Adv. Math., 264 (2014), pp.~1--54.

\bibitem{RLS}
{\sc S.~Rodr{\'{\i}}guez-L{\'o}pez and W.~Staubach}, {\em Estimates for rough
  {F}ourier integral and pseudodifferential operators and applications to the
  boundedness of multilinear operators}, J. Funct. Anal., 264 (2013),
  pp.~2356--2385.

\bibitem{RuSu}
{\sc M.~Ruzhansky and M.~Sugimoto}, {\em A local-to-global boundedness argument
  and {F}ourier integral operators}, J. Math. Anal. Appl., 473 (2019),
  pp.~892--904.

\bibitem{SSS}
{\sc A.~Seeger, C.~D. Sogge, and E.~M. Stein}, {\em Regularity properties of
  {F}ourier integral operators}, Ann. of Math. (2), 134 (1991), pp.~231--251.

\bibitem{Stein}
{\sc E.~M. Stein}, {\em Harmonic analysis: real-variable methods,
  orthogonality, and oscillatory integrals}, vol.~43 of Princeton Mathematical
  Series, Princeton University Press, Princeton, NJ, 1993.

\bibitem{Taylor}
{\sc M.~E. Taylor}, {\em Tools for {PDE}}, vol.~81 of Mathematical Surveys and
  Monographs, American Mathematical Society, Providence, RI, 2000.
\newblock Pseudodifferential operators, paradifferential operators, and layer
  potentials.

\bibitem{Trie83}
{\sc H.~Triebel}, {\em Theory of function spaces}, vol.~38 of Mathematik und
  ihre Anwendungen in Physik und Technik [Mathematics and its Applications in
  Physics and Technology], Akademische Verlagsgesellschaft Geest \& Portig
  K.-G., Leipzig, 1983.

\bibitem{Triebel4}
\leavevmode\vrule height 2pt depth -1.6pt width 23pt, {\em Theory of function
  spaces \emph{IV}.}, vol.~107 of Monographs in Mathematics, Birkh\"auser
  Verlag, Basel, 2020.

\end{thebibliography}

\end{document}